\documentclass[reqno]{amsart}
\usepackage{amsmath, amssymb, amsthm, epsfig}
\usepackage{hyperref, latexsym}
\usepackage{url}
\usepackage[mathscr]{euscript}

\usepackage{color}
\usepackage{fullpage} 
\usepackage{setspace}

\def\today{\ifcase\month\or
  January\or February\or March\or April\or May\or June\or
  July\or August\or September\or October\or November\or December\fi
  \space\number\day, \number\year}

 \newtheorem{theorem}{Theorem}
  
 \newtheorem{lemma}[theorem]{Lemma}
 \newtheorem{proposition}[theorem]{Proposition}
 \newtheorem{corollary}[theorem]{Corollary}
 \theoremstyle{definition}

 \theoremstyle{remark}

\renewcommand{\S}{\mathbb{S}}

 \newcommand{\C}{\mathbb{C}}
 \newcommand{\R}{\mathbb{R}}
 \newcommand{\N}{\mathbb{N}}

 \newcommand{\dt}{\text{\rm d}t}
 \newcommand{\du}{\text{\rm d}u}

 \newcommand{\dx}{\text{\rm d}x}

\newcommand{\dsigma}{\text{\rm d}\sigma}

\newcommand{\dxi}{\text{\rm d}\xi}

 \newcommand{\dmu}{\text{\rm d}\mu}
 \newcommand{\dnu}{\text{\rm d}\nu}

    \renewcommand{\d}{\text{\rm d}}

\newcommand{\ov}{\overline}

\begin{document}

\title[]{Some sharp restriction inequalities on the sphere}
\author[Carneiro and Oliveira e Silva]{Emanuel Carneiro and Diogo Oliveira e Silva}
\date{\today}
\address{IMPA - Instituto Nacional de Matem\'{a}tica Pura e Aplicada - Estrada Dona Castorina, 110, Rio de Janeiro, RJ, Brazil 22460-320.}
\email{carneiro@impa.br}
\address{Hausdorff Center for Mathematics, Universit\"{a}t Bonn, 53115 Bonn, Germany.}
\email{dosilva@math.uni-bonn.de}

\allowdisplaybreaks
\numberwithin{equation}{section}

\maketitle

\begin{abstract} In this paper we find the sharp forms and characterize the complex-valued extremizers of the adjoint Fourier restriction inequalities on the sphere
$$\big\|\widehat{f \sigma}\big\|_{L^{p}(\R^{d})} \lesssim \|f\|_{L^{q}(\S^{d-1},\sigma)}$$ 
in the cases $(d,p,q) = (d,2k, q)$ with $d,k \in \N$ and $q\in \R^+ \cup \{\infty\}$ satisfying: (a) $k = 2$, $q \geq 2$ and $3 \leq d\leq 7$;  \,(b) $k = 2$, $q \geq 4$ and $d \geq 8$;\, (c) $k \geq 3$, $q \geq 2k$ and $d \geq 2$. We also prove a sharp multilinear weighted restriction inequality, with weight related to the $k$-fold convolution of the surface measure. \end{abstract}

\section{Introduction}
Let $d\in\N$ and $(\mathbb{S}^{d-1},\sigma_{d-1})$ denote the $(d-1)$-dimensional unit sphere equipped with the standard surface measure $\sigma_{d-1}$. We omit the subscript on $\sigma_{d-1}$ when clear from the context and denote the total surface measure of this unit sphere by
\begin{equation}\label{Intro_area_sphere}
\omega_{d-1}:=\sigma\big(\S^{d-1}\big) = \frac{2\,\pi^{d/2}}{\Gamma(d/2)}.
\end{equation}
Given $r >0$ and $x_0 \in \R^d$, we denote by $B(x_0,r)$ the open ball of radius $r$ centered at $x_0$. If $x_0 = 0$ we simply write $B(r )$. If $f \in L^1(\S^{d-1})$, we define the Fourier transform of the measure $f\sigma$ by
\begin{equation*}\label{RN_derivative_conv}
\widehat{f\sigma}(\xi):=\int_{\mathbb{S}^{d-1}} e^{- i \zeta\cdot \xi}\, f(\zeta)\,\dsigma(\zeta)\ ; \ \  (\xi \in \R^d).
\end{equation*}
Our primary goal in this note is to find the sharp forms and characterize the extremizers of some adjoint Fourier restriction inequalities 
\begin{equation}\label{Intro_extension}
\big\|\widehat{f \sigma}\big\|_{L^{p}(\R^{d})} \lesssim \|f\|_{L^{q}(\S^{d-1})}.
\end{equation}
The full range $(d,p,q)$ for which \eqref{Intro_extension} holds is not yet fully understood, and this is the theme of the restriction conjecture in harmonic analysis (see \cite{T} for a survey on this theory). For our purposes the classical restriction theory is enough, as we consider only cases where the inequality \eqref{Intro_extension} is already established.

\medskip

Building up on the work of Christ and Shao \cite{CS, CS2}, Foschi \cite{F} recently obtained the sharp form of \eqref{Intro_extension} in the Stein-Tomas endpoint case $(d,p,q) = (3,4,2)$, showing that the constant functions are global extremizers. Here we extend this paradigm to other suitable triples $(d,p,q)$. In fact, defining 
\begin{equation*}
C(d,p,q) = \sup_{\stackrel{f \in L^{q}(\S^{d-1})}{ f \neq 0} }\frac{\big\|\widehat{f \sigma}\big\|_{L^{p}(\R^{d})}}{\|f\|_{L^{q}(\S^{d-1})}}\,,
\end{equation*} 
our first result is the following:
\begin{theorem}\label{Thm1}
Let $(d,p,q) = (d,2k, q)$ with $d,k \in \N$ and $q\in \R^+ \cup \{\infty\}$ satisfying: 
\begin{itemize}
\item[(a)]  $k = 2$, $q \geq 2$ and $3 \leq d\leq 7$;
\item[(b)]  $k = 2$, $q \geq 4$ and $d \geq 8$;
\item[(c)]  $k \geq 3$, $q \geq 2k$ and $d \geq 2$.
\end{itemize} 
Then 
\begin{equation}\label{Intro_sharp_constant}
C(d,p,q) = \omega_{d-1}^{-1/q} \ \|\widehat{\sigma}_{d-1}\|_{L^{p}(\R^{d})}.
\end{equation}
Moreover, the complex-valued extremizers of \eqref{Intro_extension} are given by
\begin{equation}\label{CV_ext}
f(\zeta) = c \,e^{i\xi\cdot \zeta},
\end{equation}
where $c \in \C\setminus \{0\}$ and $\xi \in \R^d$.
\end{theorem}
By Plancherel's theorem we have
\begin{equation}\label{Intro_Plancherel}
\| \widehat{\sigma}\|_{L^{2k}(\R^{d})} = (2\pi)^{d/2k}\,\|\sigma * \sigma * \ldots* \sigma\|_{L^{2}(\R^{d})}^{1/k},
\end{equation}
where the convolution on the right-hand side is $k$-fold. We remark that, in principle, the $k$-fold convolution of the surface measure 
$$\sigma_{d-1}^{(k)} = \sigma_{d-1} *  \ldots* \sigma_{d-1}$$ 
is a finite measure on $\R^d$ supported on the ball $\ov{B(k)}$. For $k \geq 2$, the measure $\sigma_{d-1}^{(k)}$ and the Lebesgue measure are mutually absolutely continuous on $\ov{B(k)}$ (see \cite[Eq. (2.7)]{Ch}), and we make the identification of $\sigma_{d-1}^{(k)}$ with its Radon-Nikodym derivative, which is a radial function. When $k=2$, the value of $\sigma_{d-1} * \sigma_{d-1}$ was explicitly computed in \cite[Proposition A.5]{R} as
\begin{equation}\label{Intro_conv_2}
\sigma_{d-1}\ast\sigma_{d-1}(\xi)=2^{-d+3}\, \omega_{d-2}\,  \frac{1}{|\xi|}(4-|\xi|^2)_+^{\frac{d-3}{2}},
\end{equation}
for $\xi \in \R^d$, where $x_+ = \max\{0,x\}$. We provide an alternative proof of \eqref{Intro_conv_2} in Lemma \ref{biconvsigma} below. 

\medskip

Using \eqref{Intro_area_sphere}, \eqref{Intro_Plancherel}, \eqref{Intro_conv_2} and the identity 
\begin{equation}\label{A2}
\int_0^1 t^{w-1} \,(1-t)^{z-1}\,\dt = \frac{\Gamma(w) \Gamma(z)}{\Gamma(w+z)},
\end{equation}
valid for $w,z \in \C$ with $\Re(w), \Re(z) >0$, we may simplify \eqref{Intro_sharp_constant} in the case $k=2$ to 
\begin{equation*}
C(d,4,q) = \omega_{d-1}^{-1/q + 1/4}\,\omega_{d-2}^{1/2}\,\,(2\pi)^{d/4}\,\,2^{(d-3)/4} \left[\frac{\Gamma(d-2)\,\Gamma\left(\tfrac{d-2}{2}\right)}{\Gamma\left(\tfrac{3(d-2)}{2}\right)}\right]^{1/4}.
\end{equation*}
When $k\geq 3$, we may express the value of the sharp constant in \eqref{Intro_sharp_constant} in terms of integrals involving the Bessel function of the first kind $J_v$ defined by 
$$J_v(z) =  \sum_{n=0}^{\infty} \frac{(-1)^n \big(\tfrac12 z\big)^{2n + v}}{n!\, \Gamma(v+n+1)},$$
for $v > -1$ and $\Re(z) >0$. In fact, the Fourier transform of $\sigma_{d-1}$ is given by 
\begin{equation}\label{Intro_FT_sigma}
\widehat{\sigma}_{d-1}(x)=(2\pi)^{d/2}|x|^{(2-d)/2} J_{(d-2)/2}(|x|).
\end{equation}
If we invert the $k$-th power of this Fourier transform, we find an expression for the $k$-fold convolution of the surface measure
\begin{align}\label{Intro_conv_k}
\begin{split}
\sigma_{d-1}^{(k)}(\xi) & =(2\pi)^{-d} \, (2\pi)^{dk/2} \int_{\R^d} e^{i x \cdot \xi}\,|x|^{(2-d)k/2} J_{(d-2)/2}(|x|)^k\,\dx
\end{split}
\end{align}
for $|\xi| \leq k$, provided this integral converges absolutely. 

\smallskip

Our strategy to prove Theorem \ref{Thm1} has several distinct components. Firstly, as far as the sharp inequality is concerned, we follow the outline of Christ-Shao \cite{CS, CS2} and Foschi \cite{F} (which corresponds to the case $d=3$) to prove part (a), using a spectral decomposition in spherical harmonics and the Funk-Hecke formula. We are able to extend their method up to dimension $d=7$. In order to prove parts (b) and (c) we take a different path, using a sharp multilinear weighted inequality related to the $k$-fold convolution of the surface measure (Theorem \ref{Thm2}) together with a symmetrization process over the group of rotations $SO(d)$. Secondly, as far as the characterization of the complex-valued extremizers is concerned, our main tool is a complete characterization of the solutions of the Cauchy-Pexider functional equation for sumsets of the sphere given by Theorem  \ref{Sec3_Char_ext}. This builds upon previous work by  Christ-Shao \cite{CS2} and Charalambides \cite{Ch}.

\medskip

Our second result is the following multilinear weighted adjoint restriction inequality.

\begin{theorem}\label{Thm2}
Let $d,k \geq 2$ and $f_j \in L^1(\S^{d-1})$, $f_j \neq 0$, $j =1,2,\ldots, k$. Then we have
\begin{equation}\label{Intro_weighted}
\left\|\prod_{j=1}^k \widehat{f_j\sigma}\right\|_{L^{2}(\R^d)}^2 \leq (2\pi)^{d} \int_{(\S^{d-1})^k} \sigma^{(k)}(\zeta_1 + \zeta_2 + \ldots+ \zeta_k) \,\left(\prod_{j=1}^k |f_j(\zeta_j)|^2\right)\, \dsigma(\zeta_1)\,\dsigma(\zeta_2)\, \ldots\, \dsigma(\zeta_k). 
\end{equation}
If $(d,k) \neq (2,2)$ and the right-hand side of \eqref{Intro_weighted} is finite, we have equality if and only if 
\begin{equation}\label{Intro_extremizers}
f_j(\zeta) = c_j\,e^{\nu \cdot \zeta}
\end{equation}
for $j = 1,2,  \ldots, k$, where $c_1, c_2, \ldots,c_k \in \C \setminus\{0\}$ and $\nu \in \C^d$.
\end{theorem}

\medskip

Using \eqref{Intro_conv_2}, we may specialize the inequality \eqref{Intro_weighted} to the case $k=2$ to obtain the following corollary.

\medskip

\begin{corollary}\label{Cor3}
Let $d \geq 2$ and $f_1,f_2 \in  L^1(\S^{d-1})$ with $f_1,f_2 \neq 0$. Then we have
\begin{equation}\label{Intro_weighted_case2}
\left\| \widehat{f_1\sigma} \, \widehat{f_2\sigma}\right\|_{L^{2}(\R^d)}^2 \leq (2\pi)^{d}\,2^{(-d+2)/2} \,\omega_{d-2} \int_{(\S^{d-1})^2} \!\left[\frac{(1 - \zeta_1\cdot \zeta_2)^{d-3}}{(1 + \zeta_1 \cdot \zeta_2)}\right]^{1/2}\!|f_1(\zeta_1)|^2\,|f_2(\zeta_2)|^2\, \dsigma(\zeta_1)\,\dsigma(\zeta_2). 
\end{equation}
If $ d \geq 3$ and the right-hand side of \eqref{Intro_weighted_case2} is finite, we have equality if and only if 
\begin{equation*}
f_1(\zeta) = c_1\,e^{\nu \cdot \zeta}\ \ \ {\textrm and} \ \ \ f_2(\zeta) = c_2\,e^{\nu \cdot \zeta}\,,
\end{equation*}
where $c_1, c_2 \in \C\setminus\{0\}$ and $\nu \in \C^d$.
\end{corollary}
In the case $d=2$, a version of the inequality \eqref{Intro_weighted_case2} can be found in the work of Foschi and Klainerman \cite[Example 17.5]{FK}, where it is described as ``an interesting formula''.
In the case $d=3$, the weighted inequality \eqref{Intro_weighted_case2} already appears in the work of Foschi \cite[Lemma 4.1]{F}. A novel feature here is the complete characterization of the extremizers \eqref{Intro_extremizers}. The next result is a key tool to characterize the extremizers in Theorems \ref{Thm1} and \ref{Thm2}.

\begin{theorem} \label{Sec3_Char_ext}
Let $d \geq 2$, $k \geq 2$ and $(d,k) \neq (2,2)$. Let $f_j: \S^{d-1} \to \C$ $(1 \leq j \leq k)$ and $h:\ov{B(k)} \to \C$ be measurable functions such that 
\begin{equation} \label{Sec3_Functional_Equation}
f_1(\zeta_1)f_2(\zeta_2) \ldots f_k(\zeta_k) = h(\zeta_1 + \zeta_2 +  \ldots+ \zeta_k)
\end{equation}
for $\sigma^k-$a.e. $(\zeta_1, \zeta_2,  \ldots, \zeta_k) \in (\S^{d-1})^k$. Then one the following holds:
\begin{itemize}
\item[(i)] There exist $c_1, c_2, \ldots,c_k \in \C \setminus\{0\}$ and $\nu \in \C^d$ such that 

\begin{equation*}
f_j(\zeta) = c_j\,e^{\nu \cdot \zeta}
\end{equation*}
for $\sigma-$a.e. $\zeta \in \S^{d-1}$, $j=1,2,\ldots, k$.

\smallskip

\item[(ii)] $f_j(\zeta) = 0$ for $\sigma-$a.e. $\zeta \in \S^{d-1}$, for some $j$ with $1 \leq j \leq k$.
\end{itemize}
\end{theorem}

In the spirit of Theorem \ref{Thm2}, similar multilinear weighted adjoint restriction inequalities were obtained for the paraboloid \cite{C} and cone \cite{BR}, in connection with sharp Strichartz and Sobolev-Strichartz estimates for the Sch\"{o}dinger and wave equations, respectively (in the context of the wave equation, see also \cite{KM1, KM2, KM3} for related inequalities with different `null' weights). In retrospect, the works of Kunze \cite{K},  Foschi \cite{F2} and Hundertmark-Zharnitsky \cite{HZ} were the pioneers on the existence and classification of extremizers for adjoint restriction inequalities (over the paraboloid and cone) in low dimensions. This line of research flourished and these papers were followed by a pool of very interesting works in the interface of extremal analysis and differential equations, see for instance \cite{BBCH, B, FVV, FVV2, HS, OS, OR, Q, Ra, Sh, Sh2}, in addition to the ones previously cited in this introduction.


\section{Proof of Theorem \ref{Thm2}}

\subsection{Preliminaries} Let $k \geq 2$ and $f_j \in L^1(\S^{d-1})$, $j =1,2,\ldots, k$. The convolution $(f_1\sigma * f_2\sigma * \ldots* f_k\sigma)$ is a finite measure defined on the Borel subsets $E\subset \R^d$ by 
\begin{equation*}
(f_1\sigma * f_2\sigma * \ldots* f_k\sigma)(E) = \int_{(\S^{d-1})^k} \chi_E(\zeta_1 + \zeta_2 + \ldots+ \zeta_k)\, \left(\prod_{j=1}^k\,f_j(\zeta_j)\right)\dsigma(\zeta_1)\,\dsigma(\zeta_2) \, \ldots \, \dsigma(\zeta_k).
\end{equation*}
It is then clear that this measure is supported on $\ov{B(k)}$. Since $f_j \in L^1(\S^{d-1})$, the measure  $(f_1\sigma * f_2\sigma * \ldots* f_k\sigma)$ is absolutely continuous with respect to  $\sigma^{(k)}$, and therefore it is absolutely continuous with respect to the Lebesgue measure on $\R^d$. We identify $(f_1\sigma * f_2\sigma * \ldots* f_k\sigma)$ with its Radon-Nikodym derivative with respect to the Lebesgue measure, writing it in the following way (see for instance \cite{F} or \cite[Remark 3.1]{FK})
\begin{equation}\label{Sec2_delta_form}
(f_1\sigma * f_2\sigma * \ldots* f_k\sigma)(\xi) = \int_{(\S^{d-1})^k} \delta_d(\xi - \zeta_1 - \zeta_2 - \ldots - \zeta_k)\, \left(\prod_{j=1}^k\,f_j(\zeta_j)\right)\dsigma(\zeta_1)\,\dsigma(\zeta_2)\, \ldots \, \dsigma(\zeta_k),
\end{equation}
where $\delta_d$ denotes the $d$-dimensional Dirac delta distribution. The alternative expression \eqref{Sec2_delta_form} is particularly useful in some computations, as exemplified by the next result.

\begin{lemma}\label{biconvsigma} 
Let $d\geq 2$. The surface measure $\sigma_{d-1}$ on $\S^{d-1}$ satisfies
$$\sigma_{d-1}\ast\sigma_{d-1}(\xi)=2^{-d+3}\,\omega_{d-2}\,\frac{1}{|\xi|}(4-|\xi|^2)_+^{\frac{d-3}{2}}$$
for $\xi \in \R^d$, where $x_+ = \max\{0,x\}$.
\end{lemma}

\begin{proof}
Let $\sigma:=\sigma_{d-1}$. Following \cite{F}, the surface measure on the sphere may be written as 
\begin{equation}\label{Sec2_sigma_delta}
\dsigma(\xi) = \delta(1 - |\xi|)\,\dxi = 2\,\delta(1 - |\xi|^2)\,\dxi,
\end{equation}
where $\delta$ denotes the one-dimensional Dirac delta and $\dxi$ denotes the Lebesgue measure on $\R^d$. Using \eqref{Sec2_delta_form} and \eqref{Sec2_sigma_delta} we have
\begin{align*}
\sigma\ast\sigma(\xi)&=\int_{\R^d}\delta(1-|\xi-\nu|)\,\delta(1-|\nu|)\,\dnu = \int_{\S^{d-1}}\delta(1-|\xi-\nu|)\, \dsigma(\nu)\\
&= 2\int_{\S^{d-1}}\delta\big(1-|\xi-\nu|^2\big)\,\dsigma(\nu) =2\int_{\S^{d-1}}\delta\big(2\xi\cdot\nu-|\xi|^2\big)\,\dsigma(\nu)\\
&=\frac{2}{|\xi|}\int_{\S^{d-1}}\delta\left(2\frac{\xi}{|\xi|}\cdot\nu-|\xi|\right)\,\dsigma(\nu).
\end{align*}
Passing to polar coordinates  \cite[Lemma A.5.2]{DX} in the sphere $\S^{d-1}$, we find 
\begin{align*}
\sigma\ast\sigma(\xi)&=\omega_{d-2}\,\frac{2}{|\xi|}\int_{-1}^1\delta\big(2 u-|\xi|\big)\,(1-u^2)^{\frac{d-3}{2}}\,\du\\
&=\omega_{d-2}\,\frac{1}{|\xi|}\int_{-1}^1\delta\left(u-\frac{|\xi|}{2}\right)(1-u^2)^{\frac{d-3}{2}}\,\du\\
&=\omega_{d-2}\,\frac{1}{|\xi|}\,\left(1-\frac{|\xi|^2}{4}\right)^{\frac{d-3}{2}}\,,
\end{align*}
if $|\xi|/2\in (-1,1)$. The result follows from this.
\end{proof}

\subsection{The sharp inequality} Consider the multilinear form
\begin{equation*}
T(f_1, f_2, \ldots,f_k)(\xi):= (f_1\sigma * f_2\sigma * \ldots* f_k\sigma)(\xi). 
\end{equation*}
Applying the Cauchy-Schwarz inequality in \eqref{Sec2_delta_form} with respect to the measure $\delta_d(\xi - \zeta_1 - \zeta_2 - \ldots - \zeta_k)\,\dsigma(\zeta_1)\, \ldots \, \dsigma(\zeta_k)$ we find
\begin{align}\label{Sec2_CS}
\begin{split}
|T(f_1, f_2, \ldots,f_k)(\xi)|^2 & \leq T({\bf 1}, {\bf 1}, \ldots,{\bf 1})(\xi)\,.\,T\big(|f_1|^2, |f_2|^2, \ldots\,,|f_k|^2\big)(\xi)\\
& = \sigma_{d-1}^{(k)}(\xi)\,\,T\big(|f_1|^2, |f_2|^2, \ldots \, ,|f_k|^2\big)(\xi),
\end{split}
\end{align}
where ${\bf 1}$ denotes the constant function equal to $1$. Using Plancherel's theorem, \eqref{Sec2_delta_form} and \eqref{Sec2_CS} we arrive at
\begin{align}\label{Sec2_weighted_ineq}
\begin{split}
&\left\|\prod_{j=1}^k \widehat{f_j\sigma}\right\|_{L^{2}(\R^d)}^2 = (2\pi)^d \left\|T(f_1, f_2, \ldots,f_k)\right\|_{L^{2}(\R^d)}^2 \\
& \ \ \ \ \  \leq  (2\pi)^d \int_{\R^d}  \sigma_{d-1}^{(k)}(\xi)\,\,T\big(|f_1|^2, |f_2|^2, \ldots \, ,|f_k|^2\big)(\xi)\,\dxi\\
 & \ \ \ \ \ = (2\pi)^{d}  \int_{\R^d}\int_{(\S^{d-1})^k} \sigma^{(k)}(\xi) \, \delta_d(\xi - \zeta_1 - \ldots - \zeta_k)\,\left(\prod_{j=1}^k |f_j(\zeta_j)|^2\right)\, \dsigma(\zeta_1)\,\dsigma(\zeta_2)\, \ldots\, \dsigma(\zeta_k) \,\dxi\\
& \ \ \ \ \ =  (2\pi)^{d} \int_{(\S^{d-1})^k} \sigma^{(k)}(\zeta_1 + \zeta_2 + \ldots+ \zeta_k) \,\left(\prod_{j=1}^k |f_j(\zeta_j)|^2\right)\, \dsigma(\zeta_1)\,\dsigma(\zeta_2)\, \ldots\, \dsigma(\zeta_k),
\end{split}
\end{align}
which is our desired inequality. 

\medskip

If $(d,k) \neq (2,2)$, both sides of \eqref{Sec2_weighted_ineq} are finite for $f_1 = f_2 = \ldots= f_k = {\bf 1}$, in which case we have equality. In fact, in this case, both sides are equal to $\|\widehat{\sigma}\|_{L^{2k}(\R^d)}^{2k} = (2\pi)^d \,\sigma_{d-1}^{(2k)}(0)$, which is finite by \eqref{Intro_conv_k} since $J_{v}(x) = O( |x|^{-1/2})$ as $x \to \infty$. This shows that our inequality is sharp.

\subsection{The cases of equality} \label{CE} 
Assume that the right-hand side of \eqref{Sec2_weighted_ineq} is finite and that we have equality in \eqref{Sec2_weighted_ineq}. Then the right-hand side of \eqref{Sec2_CS} is finite and we have equality in \eqref{Sec2_CS} for all $\xi$ in a subset $A_1 \subset \ov{B(k)}$ of full Lebesgue measure (note that both sides of \eqref{Sec2_CS} are zero for $\xi \notin \ov{B(k)}$).

\smallskip

Let $\xi \in A_1$ and consider the singular measure on $(\S^{d-1})^k$ given by 
\begin{equation}\label{Def_Psi}
\d \Psi_{\xi} = \delta_d(\xi - \zeta_1 - \zeta_2 - \ldots - \zeta_k)\,\dsigma(\zeta_1)\, \ldots \, \dsigma(\zeta_k).
\end{equation}
The condition of equality in the Cauchy-Schwarz inequality tells us that there exists a function $h$ on $A_1$ such that 
\begin{equation}\label{Sec2_Cond_eq_Sigma}
f_1(\zeta_1)f_2(\zeta_2) \ldots f_k(\zeta_k) = h(\xi)
\end{equation}
for $\Psi_{\xi}$-a.e. $(\zeta_1, \zeta_2, \ldots, \zeta_k) \in (\S^{d-1})^k$. If this is the case, then
\begin{equation*}
h(\xi) = \frac{\int_{(\S^{d-1})^k} f_1(\zeta_1)f_2(\zeta_2) \ldots f_k(\zeta_k)\,\d\Psi_{\xi}(\zeta_1,\zeta_2, \ldots, \zeta_k)}{\int_{(\S^{d-1})^k} \d\Psi_{\xi}(\zeta_1,\zeta_2, \ldots, \zeta_k)} = \frac{(f_1\sigma * f_2\sigma * \ldots* f_k\sigma)(\xi)}{\sigma^{(k)}(\xi)}\,,
\end{equation*}
and we see that $h$ is a Lebesgue measurable function. Note that $\sigma^{(k)}(\xi) >0$ for all $|\xi| < k$ (this follows from the explicit evaluation in Lemma \ref{biconvsigma} and induction on $k$) and we might have $\sigma^{(k)}(\xi) = +\infty$ only on a set of Lebesgue measure zero. Consider the set
\begin{equation*}
E = \left\{(\zeta_1,\zeta_2, \ldots, \zeta_k) \in (\S^{d-1})^k; \ \ f_1(\zeta_1)f_2(\zeta_2) \ldots f_k(\zeta_k) \neq h(\zeta_1 + \zeta_2 +  \ldots+ \zeta_k)\right\}
\end{equation*}
and let $\sigma^{k}$ denote the product measure on $(\S^{d-1})^k$. We claim that $\sigma^k (E) = 0$. In fact,
\begin{align*}\label{Sec2_Cal_E_Etilde}
\begin{split}
\sigma^k (E) &= \int_{(\S^{d-1})^k} \chi_{E}(\zeta_1, \ldots, \zeta_k) \,\dsigma(\zeta_1)\, \ldots\, \dsigma(\zeta_k)\\
& =  \int_{(\S^{d-1})^k} \int_{\R^d} \chi_{E}(\zeta_1, \ldots, \zeta_k)\,\delta_d(\xi - \zeta_1 - \zeta_2 - \ldots - \zeta_k)\,\dxi\,\dsigma(\zeta_1)\, \ldots\, \dsigma(\zeta_k)\\
& = \int_{\R^d} \int_{(\S^{d-1})^k}  \chi_{E}(\zeta_1, \ldots, \zeta_k)\, \d\Psi_{\xi} (\zeta_1, \ldots, \zeta_k)\, \dxi\\
& = 0.
\end{split}
\end{align*}
Therefore we must have, for this $h: \ov{B(k)} \to \C$ measurable,
\begin{equation} \label{Sec2_Functional_Equation}
f_1(\zeta_1)f_2(\zeta_2) \ldots f_k(\zeta_k) = h(\zeta_1 + \zeta_2 +  \ldots+ \zeta_k)
\end{equation}
for $\sigma^k-$a.e. $(\zeta_1, \zeta_2,  \ldots, \zeta_k) \in (\S^{d-1})^k$.

\medskip

On the other hand, if \eqref{Sec2_Functional_Equation} holds, we can reverse all the steps above to conclude that we have equality a.e. in \eqref{Sec2_CS} and thus equality in \eqref{Sec2_weighted_ineq} (possibly with both sides being infinity).

\subsection{Characterization of the extremizers}

The characterization of the functions $f_j:\S^{d-1} \to \C$ \,($1 \leq j \leq k$) that satisfy the functional equation \eqref{Sec2_Functional_Equation} is given by Theorem \ref{Sec3_Char_ext}, whose proof we postpone to the next section. Assuming this result, let us conclude the proof of Theorem \ref{Thm2}.

\medskip

If the right-hand side of \eqref{Intro_weighted} is finite and we have equality, \eqref{Sec2_Functional_Equation} and Theorem \ref{Sec3_Char_ext} show that \eqref{Intro_extremizers} must hold (recall that we are assuming $f_j \neq 0$). Conversely, if \eqref{Intro_extremizers} holds we have that the right-hand side of \eqref{Intro_weighted} is finite (since all $f_j$'s are uniformly bounded) and, as observed in the previous subsection, we find that equality occurs in \eqref{Intro_weighted}.


\section{Proof of Theorem \ref{Sec3_Char_ext} - revisiting the work of Charalambides}

\subsection{Preliminaries} In the work \cite{Ch}, M. Charalambides developed a thorough study of the solutions of the Cauchy-Pexider functional equation \eqref{Sec3_Functional_Equation} (for sumsets of general submanifolds $M \subset \R^d$), building up on the work initiated by Christ and Shao \cite{CS2} for the sphere $\S^2 \subset \R^3$. In particular, Charalambides \cite{Ch} establishes the following result:

\begin{lemma}\label{Lem6_Ch_result}
{\rm (} \!\!{\rm cf.} \cite[Theorems 1.2, 1.6 and 1.8]{Ch}\! \!{\rm)}
Theorem \ref{Sec3_Char_ext} is true in the cases $k=2$; $d \geq 3$ and $k=3$; $d \geq 2$, under the additional assumption that $\sigma\big(f_j^{-1}(\{0\})\big) = 0$, for all $j =1,2, \ldots, k$.
\end{lemma}

\noindent Our work in this section is to {\it remove the assumption} $\sigma\big(f_j^{-1}(\{0\})\big) = 0$, $j =1,2, . . ., k$, in the case of the sphere $\S^{d-1}$, and to extend this result to higher $k$. We start with a lemma that essentially follows from the work of Charalambides. We only indicate the main modifications needed with respect to the corresponding proof in \cite{Ch}. 

 \begin{lemma}\label{Lem7}
{\rm (} \!\!{\rm cf.} \cite[Proof of Theorem 1.2]{Ch} \!\!{\rm)}  For $d \geq 3$, let $f_1,f_2: \S^{d-1} \to \C$ and $h:\ov{B(2)} \to \C$ be measurable functions satisfying 
\begin{equation}\label{Sug_Diogo_Sec3}
f_1(\zeta_1)f_2(\zeta_2) = h(\zeta_1 + \zeta_2 )
\end{equation}
 for $\sigma^2-$a.e. $(\zeta_1, \zeta_2) \in (\S^{d-1})^2$. Then for each $\xi \in B(2) \setminus\{0\}$ there exists a ball $B_{\xi} = B(\xi,r_{\xi}) \subset B(2) \setminus\{0\}$ and a measurable function $H_{\xi}: B_{\xi} + B_{\xi} \to \C$ such that 
\begin{equation}\label{hH}
h(x) \,h(y) = H_{\xi}(x+y)
\end{equation}
for Lebesgue a.e. $(x,y) \in B_{\xi} \times B_{\xi}$. 
\end{lemma}

\begin{proof} We briefly recall the notation used in \cite{Ch}. Let $M = \S^{d-1}$ and write points in $(\R^d)^4 \times (\R^d)^4$ as $(x,y)$, where $x = (x_1,x_2,x_3,x_4)$ and $y = (y_1, y_2, y_3, y_4)$, so that $x_j, y_j \in \R^d$ for $1 \leq j \leq 4$. Let $\Pi$ be the hyperplane in $(\R^d)^4 \times (\R^d)^4$ defined by $x_1 + y_2 = x_3 + y_4$ and $y_1 + x_2 = y_3 + x_4$ and let $\mathcal{P}_M = (M^4 \times M^4) \cap \Pi$. Let $\mathcal{S}_M$ be the set of smooth points of $\mathcal{P}_M$, i.e. the points where $M^4 \times M^4$ intersects $\Pi$ transversally, and let $\Lambda \subset (\R^d)^4$ be the $3d$-dimensional hyperplane given by the points $(w_1, w_2, w_3,w_4) \in  (\R^d)^4$ such that $w_1 + w_2 = w_3 + w_4$. The linear addition map $(\R^d)^4 \times (\R^d)^4 \mapsto (\R^d)^4$ given by $(x,y) \mapsto x+y$ restricts to a smooth map $\pi_M : \mathcal{S}_M \to \Lambda$ and we call $\mathcal{R}_M$ the set of regular points of $\pi_M$, i.e. the points of $\mathcal{S}_M$ where $\pi_M$ is a submersion. Finally, let 
$$R_M = \{ x + y; \ x,y \in M; \ x \neq \pm y\} = B(2) \setminus \{0\}.$$

\smallskip

The crux of the matter here, given $\xi = z_1 \in R_M = B(2) \setminus \{0\}$, is to choose a point $z = (z_1,z_2,z_3,z_4)$, with $z_2 = z_1$, such that $z = \pi_M (x,y)$ for some $(x,y) \in \mathcal{R}_M$ (see \cite[Proof of Theorem 1.2, p. 239, line 6]{Ch}). If this choice can be made, the lemma follows from the argument in \cite[Proof of Theorem 1.2, p. 239, lines 6 - 23]{Ch}.

\smallskip

Since $z_1 \in R_M = B(2) \setminus \{0\}$ we start by choosing freely $x_1, y_1 \in \S^{d-1}$ such that 
$$z_1 = x_1 + y_1.$$
Note that this implies that $x_1 \neq \pm y_1$. Now choose $x_2$ and $y_2$, in a way that $x_2,y_2 \neq \pm x_1, \pm y_1$ and such that 
$$z_2 = x_2 + y_2 = z_1.$$
Note again that $x_2 \neq \pm y_2$. By \cite[Lemma 2.3]{Ch}, it already follows that the point $(x,y)$ that we are constructing belongs to $\mathcal{S}_M$, and \cite[Eq. (2.4)]{Ch} is partially fulfilled. Note also that 
$${\rm span}\{x_1, y_1\} \cap {\rm span}\{x_2, y_2\} = {\rm span}\{z_1\}.$$
Now we are relatively free to choose $x_3, x_4, y_3,y_4$. In fact these must satisfy
$$x_1 + y_2 = x_3 + y_4$$
and
$$y_1 + x_2 = y_3 + x_4,$$
which are the equations defining $\Pi$, and we must complete the conditions \cite[Eq. (2.4) and (2.5)]{Ch} in order to guarantee that the point $(x,y)$ belongs to $\mathcal{R}_M$. Note that $x_1 + y_2$ and $y_1 + x_2$ both belong to $R_M = B(2) \setminus \{0\}$. We can choose, for instance, $x_3$ close (but not equal) to $x_1$ and $y_4$ close (but not equal) to $y_2$, and similarly, $y_3$ close (but not equal) to $x_2$ and $x_4$ close (but not equal) to $y_1$. Therefore we can assure that $x_j \neq \pm y_j$ for all $j =1,2,3,4$ (this establishes  \cite[Eq. (2.4)]{Ch}) and ${\rm span}\{x_3, y_3\} \cap {\rm span}\{x_4, y_4\}$ is close to  ${\rm span}\{x_1, x_2\} \cap {\rm span}\{y_1, y_2\}$, which is a line different from ${\rm span}\{z_1\}$, thus leading to 
$${\rm span}\{x_1, y_1\} \cap {\rm span}\{x_2, y_2\} \cap {\rm span}\{x_3, y_3\} \cap {\rm span}\{x_4, y_4\} = \{0\}.$$
This completes \cite[Eq. (2.5)]{Ch}, which shows that $(x,y) \in \mathcal{R}_M$, and that 
$$\pi_M(x,y) = (z_1, z_1, z_3, z_4),$$
where $z_3 = x_3 + y_3$ and $z_4 = x_4 + y_4$.
\end{proof}

\subsection{Proof of Theorem \ref{Sec3_Char_ext}}

Throughout this proof we denote by $\lambda^d$ the $d$-dimensional Lebesgue measure.

\subsubsection{The case $k=2$ and $d \geq 3$} .
\medskip

\noindent {\it Step 1. Local argument.} Fix $\xi \in B(2)  \setminus\{0\}$ and let $B_{\xi}$ and $H_{\xi}$ as in Lemma \ref{Lem7}. We claim that $h(x) = c_{\xi}\, e^{\nu_{\xi} \cdot x}$ a.e. in $B_{\xi}$, for some $c_{\xi} \in \C$ and $\nu_{\xi} \in \C^d$.

\medskip

If $\lambda^d\big(h^{-1}(\{0\}) \cap B_{\xi}\big)=0$, we may use \cite[Lemma 2.1]{Ch} to reach the desired conclusion. 

\medskip

If $\lambda^d\big(h^{-1}(\{0\}) \cap B_{\xi}\big)>0$, we will be done if we prove that we can choose $c_{\xi}=0$. Suppose this is not the case. Let $A_1:=h^{-1}(\{0\}) \cap B_{\xi}$ and define $A_2:=B_{\xi} \setminus A_1$. Assuming $\lambda^d(A_1)>0$ and $\lambda^d(A_2)>0$, we aim at a contradiction.

\medskip

For a.e. $x\in A_1$, identity \eqref{hH} holds for a.e. $y\in B_{\xi}$ (this is a consequence of Fubini's theorem). Similarly, for a.e. $x\in A_2$, identity \eqref{hH} holds for a.e. $y\in A_2$. Let $\widetilde{A}_1, \widetilde{A}_2$ denote the full measure subsets of $A_1, A_2$, respectively, for which these conclusions hold. Then, given $\epsilon>0$, there exist $x_1\in\widetilde{A}_1$ and $ x_2\in\widetilde{A}_2$ such that $|x_1-x_2|<\epsilon$ (the existence of such $x_1, x_2$ is guaranteed by the hypotheses $\lambda^d(A_1), \lambda^d(A_2)>0$). Now, by the definition of $\widetilde{A}_1$, we conclude that $H_{\xi}\equiv 0$ a.e. on $x_1+B_{\xi}$. By the definition of $\widetilde{A}_2$, we conclude that $H_{\xi}\neq 0$ a.e. on $x_2+A_2$. However, for sufficiently small $\epsilon>0$, 
$$\lambda^d\big((x_1+B)\cap(x_2+A_2)\big)>0,$$ 
and we reach a contradiction. The conclusion is that, if $\lambda^d(A_1)>0$, then $\lambda^d(A_2)=0$ and thus $h\equiv 0$ a.e. on $B_{\xi}$.

\medskip

\noindent {\it Step 2. Local-to-global argument.} Take $\xi_0 \in  B(2)\setminus\{0\}$. From the previous step we know that there exist $c_{0} \in \C$ and $\nu_{0} \in \C^d$ such that $h(x) = c_{0}\, e^{\nu_{0} \cdot x}$ a.e. in $B_{\xi_0}$. Consider the set
\begin{equation*}
\Omega:=\{z\in B(2)\setminus\{0\}: h(x)=c_0 \,e^{\nu_0\cdot x}\textrm{ a.e. in a neighborhood of }z\}.
\end{equation*}
By construction, $\Omega$ is an open subset of $B(2)\setminus\{0\}$. We claim that $\Omega$ is also closed in $B(2)\setminus\{0\}$. To see this, suppose not, and take a point $\xi\in\overline{\Omega}\setminus\Omega$ (the closure is taken in $B(2)\setminus\{0\}$). Since $\xi\in B(2)\setminus\{0\}$, there exists an open ball $B_{\xi} = B(\xi,r_\xi)$ on which $h(x)=c_\xi \,e^{\nu_\xi\cdot x}$ almost everywhere. Since $\xi\in\overline{\Omega}$, the intersection $\Omega\cap B_{\xi}$ is nonempty. Take $z\in\Omega\cap B_{\xi}$. Then, since $z\in\Omega$, the identity $h(x)=c_0 \,e^{\nu_0\cdot x}$ holds almost everywhere in a sufficiently small ball $B_z\subset B_{\xi}$. Now, if $c_0=0$, then $c_\xi=0$ and $\xi\in\Omega$, a contradiction. If $c_0\neq 0$, it follows that $c_\xi=c_0$ and $\nu_\xi=\nu_0$ (this can be seen by differentiating the identity $c_\xi c_0^{-1}=e^{(\nu_\xi-\nu_0)\cdot x}$ with respect to the variable $x$). The conclusion is, again, that $\xi\in\Omega$, an absurd. We then conclude that $\Omega$ is closed in $B(2)\setminus\{0\}$ and, since this is a connected set, it follows that $\Omega=B(2)\setminus\{0\}$. 

\medskip

An application of the Lebesgue differentiation theorem yields $h(x)=c_0 \,e^{\nu_0\cdot x}$ a.e. in $B(2)\setminus\{0\}$.

\medskip

\noindent {\it Step 3. Conclusion in the case $k=2$ and $d \geq 3$.} We now achieve the conclusion for $f_1$ and $f_2$. Let us split the analysis in two cases:

\medskip

If $c_0\neq 0$, we claim that $\sigma\big(f_1^{-1}(\{0\})\big)=\sigma\big(f_2^{-1}(\{0\})\big)=0$ and the conclusion follows from Lemma \ref{Lem6_Ch_result}. In fact, if this were not the case, assume without loss of generality that $A_1 = f_1^{-1}(\{0\})$ is such that $\sigma(A_1) >0$. Let $Q$ be the set of pairs $(\zeta_1,\zeta_2) \in (\S^{d-1})^2$ for which \eqref{Sug_Diogo_Sec3} does not hold. By assumption $\sigma^2(Q) = 0$. Let $E = \{\zeta_1 + \zeta_2;\ \zeta_1 \in A_1,\, \zeta_2 \in \S^{d-1},\, (\zeta_1, \zeta_2) \notin Q\}$. Observe that 
\begin{align*}
\sigma * \sigma(E) & = \sigma^2\big\{ (w_1,w_2) \in (\S^{d-1})^2; \ w_1 +  w_2 \in E\big\}\\
& \geq \sigma^2\big\{ (\zeta_1,\zeta_2) \in (\S^{d-1})^2; \ \zeta_1 \in A_1,\, \zeta_2 \in \S^{d-1},\, (\zeta_1, \zeta_2) \notin Q\big\}\\
& = \sigma^2\big\{ (\zeta_1,\zeta_2) \in (\S^{d-1})^2; \ \zeta_1 \in A_1,\, \zeta_2 \in \S^{d-1}\big\}\\
& = \sigma(A_1)\, \sigma(\S^{d-1})\\
& >0.
\end{align*}
As noted in the introduction, the measures $\sigma\ast\sigma$ and $\lambda^d$ are mutually absolutely continuous on $B(2)$, and so we find that $h\equiv 0$ on a subset of $B(2)$ of positive $\lambda^d-$measure (namely $E$), a contradiction.

\medskip

If $c_0=0$, let $E_j = f_j^{-1}(\C \setminus \{0\})$ for $j=1,2$. We claim that we cannot have $\sigma(E_j)>0$ for $j=1,2$. In fact, if this were the case, arguing as above, the sumset $E_1 + E_2$ would have positive $\lambda^d-$measure and $h$ would be non-zero on a subset of $B(2)$ of positive $\lambda^d-$measure, a contradiction. Therefore, we must have $\sigma(E_j)=0$ for at least one $j$. This possibility falls under the item (ii) of Theorem \ref{Sec3_Char_ext}. 

\subsubsection{The case $k\geq3$ and $d \geq 3$} .
\medskip

\noindent {\it Step 4. Induction argument.} To extend the previous result for $k \geq 3$ in dimension $d\geq 3$, we proceed by induction on the degree of the multilinearity $k$.  We  start by showing how the trilinear case $k=3$ can be deduced from the case $k=2$. Suppose that 

\begin{equation}\label{trilinear}
f_1(\zeta_1)\,f_2(\zeta_2)\,f_3(\zeta_3)=h(\zeta_1+\zeta_2+\zeta_3)
\end{equation}
holds $\sigma^3-$a.e. on $(\S^{d-1})^3$. Then for $\sigma-$a.e. $\zeta_1\in \S^{d-1}$, identity \eqref{trilinear} holds for $\sigma^2-$a.e. $(\zeta_2,\zeta_3)\in (\S^{d-1})^2$. We split the analysis in two cases:

\medskip

If $f_j\equiv 0$ $\sigma-$a.e. for some $j$ with $1\leq j \leq 3$, we are done.

\medskip

Otherwise, let $E_1 = f_1^{-1}(\C \setminus \{0\})$. Then $\sigma(E_1) >0$. Choose $ z \in E_1$ for which
identity \eqref{trilinear} holds with $\zeta_1= z$ for $\sigma^2-$a.e. $(\zeta_2,\zeta_3)\in (\S^{d-1})^2$. Then, defining  $\widetilde{h}_{z}(\zeta):={f_1(z)^{-1}}\,h(\zeta+z)$, we have that
\begin{equation*}
f_2(\zeta_2)\,f_3(\zeta_3)=\widetilde{h}_{z}(\zeta_2+\zeta_3)\
\end{equation*}
for $\sigma^2-$a.e. $(\zeta_2,\zeta_3)\in (\S^{d-1})^2$. By the case $k=2$, there exist $c_2,c_3\in \C\setminus\{0\}$ and $\nu\in\C^d$ such that
\begin{equation*}
f_2(\zeta_2)=c_2 \,e^{\nu\cdot \zeta_2}\textrm{ and } f_3(\zeta_3)=c_3 \,e^{\nu\cdot \zeta_3}.
\end{equation*}
Repeating this argument for $f_2$ instead of $f_1$, we conclude that $f_1(\zeta_1)=c_1 \,e^{\nu\cdot \zeta_1}$ for some $c_1 \in \C\setminus\{0\}$ and the same $\nu \in\C^d$. The general $k-$linear case follows similarly by induction.

\subsubsection{The case $k = 3$ and $d = 2$} .
\medskip

\noindent {\it Step 5. Revisiting the argument of Charalambides in \cite[Section 5]{Ch}.} We now deal with the case of three functions $f_1, f_2, f_3$ on the circle $\S^1$. 

\medskip

If $f_j\equiv 0$ $\sigma-$a.e. for some $j$ with $1\leq j \leq 3$, we are done. 

\medskip

So assume that for every $1 \leq j \leq 3$ we have $\sigma\big(f_j^{-1}( \C\setminus \{0\})\big) > 0$. We shall prove that in this case we must have $\sigma\big(f_j^{-1}(\{0\})\big) = 0$ for $1 \leq j \leq 3$, and we will be done by Lemma \ref{Lem6_Ch_result}. This follows from the arguments in \cite[Section 5]{Ch} modulo some adjustments. First, let us define $\gamma: I = (0,2\pi) \to \R^2$ by 
\begin{equation*}
\gamma(x) = (\cos x, \sin x),
\end{equation*} 
which is a unit speed parametrization of the circle $\S^1$ (here we are excluding a point, but this is harmless) by the open interval $I$. Writing $F_{j}(x) = |f_j(\gamma(x))|$, we have the functional equation
\begin{equation}\label{Sec4_Fun_Eq_mod}
F_1(x_1)\, F_2(x_2)\, F_3(x_3) = H\left(\sum_{j=1}^3 \gamma(x_j)\right)
\end{equation}
for $\lambda^3-$ a.e. $(x_1, x_2, x_3) \in I^3$, with $H(z) = |h(z)|$. 

\medskip

We first show that all the $F_j$'s are bounded $\lambda-$a.e. in $I$, and therefore $H$ is also bounded $\lambda^3-$a.e. in $B(3)$. In fact, by hypothesis, the set $\{(x_2,x_3) \in I^2; F_2(x_2)F_3(x_3) >0\}$ has positive $\lambda^2$-measure, and we can choose $N$ large enough such that 
$$K = \{(x_2,x_3) \in I^2; N^{-1} < F_2(x_2)F_3(x_3) < N\}$$
verifies $\lambda^2(K) >0$. We may therefore choose a point $(u_2, u_3)$ that belongs to $K$, to the Lebesgue set of the characteristic function $\chi_{K}$, and such that ${\rm span}\,\{\gamma'(u_2), \gamma'(u_3)\} = \R^2$. Let $z = \gamma(u_2) + \gamma(u_3)$, and choose neighborhoods $U$ of $(u_2, u_3)$ and $V = B(z,r)$ such that the map $\beta: I^2 \to \R^2$ given by $\beta(x_2, x_3) = \gamma(x_2) + \gamma(x_3)$ is diffeomorphism from $U$ onto $V$. Let $K_1 = K \cap U$ and note that $\lambda^2 (K_1) >0$ (since $(u_2,u_3)$ is a Lebesgue point of $K$). Now let $I_1 \subset I$ be such that $\lambda(I \setminus I_1) =0$ and for each $x_1 \in I_1$ the functional equation \eqref{Sec4_Fun_Eq_mod} is satisfied at the point $(x_1, x_2, x_3)$ for $\lambda^2-$a.e. $(x_2, x_3) \in I^2$.

\medskip

By Lusin's theorem applied to $H|_{\gamma(I) + V}$, we may find a compact subset $T$ of the open set $\gamma(I) + V$ and a constant $C < \infty$ such that 
\begin{equation}\label{Lusin1}
\lambda^2((\gamma(I) + V)\setminus T) < \lambda^2(\beta(K_1)) \neq 0
\end{equation}
and for all $w \in T$ we have $H(w) \leq C$. Let $x_1 \in I_1$. Since $\lambda^2(\gamma(x_1) + \beta(K_1)) = \lambda^2(\beta(K_1))$, we find by \eqref{Lusin1} that $\lambda^2\big( T \cap (\gamma(x_1) + \beta(K_1))\big) >0$. We conclude that there exists $(x_2,x_3) \in K_1$ such that $H(\gamma(x_1) + \beta(x_2, x_3)) \leq C$ and \eqref{Sec4_Fun_Eq_mod} holds at $(x_1,x_2,x_3)$, leading to 
\begin{equation*}
F_1(x_1) \leq N\,C.
\end{equation*}
This proves that $F_1$ is bounded $\lambda-$a.e. in $I$. We may apply the same argument to $F_2$ and $F_3$.

\medskip

Having constructed the diffeomorphism $\beta: U \to V$ above, and since the point $(u_2, u_3)$ belongs to $K$ (and is a Lebesgue point of $K$) we can pick a small ball $U' \subset U$ around $(u_2, u_3)$ to see that 
\begin{equation*}
\int_{U} F_2(x_2)\, F_3(x_3) \,\dx_2\, \dx_3 \geq \int_{U'} F_2(x_2)\, F_3(x_3) \,\dx_2\, \dx_3 >0.
\end{equation*}
Following the outline of \cite[Section 5]{Ch} we show that the function $F_1$ is equal $\lambda-$a.e. to a differentiable function, and by analogy so are $F_2$ and $F_3$, and thus $H$. From now on we make these identifications.

\medskip

For every $x_1 \in I$ such that $F_1(x_1) \neq 0$ we argue as in \cite[Section 5, Eq. (5.5)]{Ch} (here we might have to make a new choice of the neighborhood $U$ in order to have $F_2(x_2)F_3(x_3) \neq 0$ for $(x_2,x_3) \in U$) to conclude that there is a neighborhood of $B \subset I$ of $x_1$ such that $F_1(x) = c_B \,e^{\nu_{B}\cdot \gamma(x)}$ for all $x \in B$, where $c_B \in \C\setminus \{0\}$ and $\nu_B \in \C^2$. We now argue as in Step 2, to conclude that for every connected component $W$ of the set $\{x \in I;\ F_1(x) \neq 0\}$ we must have $F_1(x) = c_W \,e^{\nu_{W}\cdot \gamma(x)}$ for $x \in W$. Since $F_1$ is continuous, this plainly implies that either $F_1 \equiv 0$ (which is not the case by hypothesis) or $F_1(x) \neq 0$ for all $x \in I$, which is the desired conclusion. The same holds for $F_2$ and $F_3$, and we conclude the proof of our original claim, i.e. that $\sigma\big(f_j^{-1}(\{0\})\big) = 0$ for $1 \leq j \leq 3$.

\subsubsection{The case $k\geq 4$ and $d = 2$} .
\medskip

\noindent {\it Step 6. Induction argument for $d = 2$.} In order to prove Theorem \ref{Sec3_Char_ext} in dimension $d=2$ for $k \geq 4$ we proceed by induction as in the Step 4. 

\medskip

This concludes the proof of Theorem \ref{Sec3_Char_ext}.


\section{Proof of Theorem \ref{Thm1} - cases (b) and ( \!\!c): symmetrization over $SO(d)$}

\subsection{Reduction to the case $L^{2k}$ to $L^{2k}$} \label{Holder reduction}

If we prove Theorem \ref{Thm1} in the case $(d,2k,2k)$, for  $d,k \geq 2$ with $(d,k) \neq (2,2)$, the corresponding cases $(d,2k,q)$ for $q >2k$ follow directly. In fact, by H\"{o}lder's inequality we have
\begin{equation}\label{Sec5_eq1.0}
\big\|\widehat{f \sigma}\big\|_{L^{2k}(\R^{d})} \leq C(d,2k,2k)\, \|f\|_{L^{2k}(\S^{d-1})} \leq  C(d,2k,2k)\, \omega_{d-1}^{\frac{1}{2k} - \frac{1}{q}} \|f\|_{L^{q}(\S^{d-1})} = C(d,2k,q)\,\|f\|_{L^{q}(\S^{d-1})}.
\end{equation}
In order to have equality in \eqref{Sec5_eq1.0}, we must have equality in the leftmost inequality, which happens only for the functions given by \eqref{CV_ext}. Since the functions in \eqref{CV_ext} have constant absolute value, we also have equality in H\"{o}lder's inequality, and thus in \eqref{Sec5_eq1.0}.

\subsection{$L^{2k}$ to $L^{2k}$ inequality} Let $d,k \geq 2$ with  $(d,k) \neq (2,2)$, and write
\begin{equation*}
K(\zeta_1, \zeta_2, \ldots,\zeta_k) = \sigma^{(k)}(\zeta_1 + \zeta_2 + \ldots + \zeta_k),
\end{equation*}
where $\zeta_j \in \S^{d-1}$ for $1 \leq j \leq k$. From Theorem \ref{Thm2} we have
\begin{align}\label{kfirststep} 
\big\|\widehat{f\sigma}\big\|_{L^{2k}(\R^d)}^{2k} \leq (2\pi )^d\,\int_{(\S^{d-1})^k}  |f(\zeta_1)|^2\ldots |f(\zeta_k)|^2\,K(\zeta_1, \ldots ,\zeta_k)\,\dsigma(\zeta_1)\ldots \dsigma(\zeta_k),
\end{align}
with nontrivial equality if and only if 
\begin{equation}\label{Sec5_equality}
f(\zeta) = c\,e^{\nu \cdot \zeta},
\end{equation}
with $c \in \C \setminus \{0\}$ and $\nu \in \C^d$. Let $SO(d)$ denote the special orthogonal group, i.e. the orthogonal $d \times d$ real matrices with determinant $1$. Note that the surface measure $\dsigma$ is invariant under the action of $SO(d)$ and that our kernel $K$ verifies
\begin{equation*}
K(R\zeta_1, R\zeta_2, \ldots,R\zeta_k)  = K(\zeta_1, \zeta_2, \ldots,\zeta_k),
\end{equation*}
for every $\zeta_1, \zeta_2, \ldots, \zeta_k \in \S^{d-1}$ and $R \in SO(d)$. Equipping the compact group $SO(d)$ with its normalized Haar measure $\dmu$, we can rewrite the integral on the right-hand side of \eqref{kfirststep} as
\begin{align}\label{Sec5_H_1}
\begin{split}
& \int_{SO(d)}\left(\int_{(\S^{d-1})^k} |f(R \zeta_1)|^2\ldots |f(R \zeta_k)|^2 \, K(\zeta_1,\ldots,\zeta_k)\,\dsigma(\zeta_1)\ldots \dsigma(\zeta_k)\right) \dmu(R)\\
& \ \ \ \ \ \ \ =  \int_{(\S^{d-1})^k}\left(\int_{SO(d)} |f(R \zeta_1)|^2\ldots |f(R \zeta_k)|^2 \,\dmu(R)\right)K(\zeta_1,\ldots,\zeta_k) \,\dsigma(\zeta_1)\ldots \dsigma(\zeta_k).
\end{split}
\end{align}
The inner integral can be estimated with H\"{o}lder's inequality:
\begin{equation}\label{Sec5_Holder}
\int_{SO(d)} |f(R \zeta_1)|^2\ldots |f(R \zeta_k)|^2 \,\dmu(R) \leq \prod_{j=1}^k \left(\int_{SO(d)} |f(R \zeta_j)|^{2k}\,\dmu(R)\right)^{1/k}.
\end{equation}
Note that
\begin{align}\label{Sec5_H_2}
\begin{split}
\int_{SO(d)} |f(R \zeta)|^{2k}\,\dmu(R) & = \frac{1}{\omega_{d-1}} \int_{\S^{d-1}} \int_{SO(d)} |f(R \zeta)|^{2k}\,\dmu(R) \,\dsigma(\zeta) \\
& =  \frac{1}{\omega_{d-1}} \int_{SO(d)}\int_{\S^{d-1}}  |f(R \zeta)|^{2k}\dsigma(\zeta)\, \dmu(R)\\
& = \frac{1}{\omega_{d-1}} \,\|f\|_{L^{2k}(\S^{d-1})}^{2k},
\end{split}
\end{align}
for any $\zeta \in \S^{d-1}$. From \eqref{kfirststep}, \eqref{Sec5_H_1}, \eqref{Sec5_Holder} and \eqref{Sec5_H_2} we arrive at
\begin{align}\label{Sec5_rest_ineq}
\begin{split}
\big\|\widehat{f\sigma}\big\|_{L^{2k}(\R^d)}^{2k} & \leq (2\pi )^d\, \omega_{d-1}^{-1}\, \|f\|_{L^{2k}(\S^{d-1})}^{2k} \int_{(\S^{d-1})^k} K(\zeta_1,\ldots,\zeta_k) \,\dsigma(\zeta_1)\ldots \dsigma(\zeta_k)\\
& = (2\pi )^d\, \omega_{d-1}^{-1} \, \sigma^{(2k)}(0) \,\|f\|_{L^{2k}(\S^{d-1})}^{2k}\\
& = \omega_{d-1}^{-1} \, \|\widehat{\sigma}\|_{L^{2k}(\R^d)}^{2k} \,\|f\|_{L^{2k}(\S^{d-1})}^{2k},
\end{split}
\end{align}
which is our desired sharp inequality. 

\subsection{Cases of equality} In order to have nontrivial equality in \eqref{Sec5_rest_ineq}, on top of condition \eqref{Sec5_equality}, we must have equality in \eqref{Sec5_Holder} for $\sigma^k-$a.e. $(\zeta_1, \ldots, \zeta_k) \in (\S^{d-1})^k$. This implies that for $\sigma^k-$a.e. $(\zeta_1, \ldots, \zeta_k) \in (\S^{d-1})^k$ we must have
\begin{equation}\label{Sec5_Holder_cond_eq}
a_1\, |f(R\zeta_1)| = a_2\, |f(R\zeta_2)|  = \ldots = a_k\, |f(R\zeta_k)|,
\end{equation}
for $\mu-$a.e. $R \in SO(d)$, where $a_j >0$ for $1 \leq j \leq k$. If we integrate \eqref{Sec5_Holder_cond_eq} over the group of rotations $SO(d)$ we see that the $a_j$'s must be equal. Using \eqref{Sec5_equality} we claim that $\nu \in \C^d$ must be purely imaginary. If not, we could take a point $(\zeta_1, \ldots, \zeta_k) \in (\S^{d-1})^k$ for which \eqref{Sec5_Holder_cond_eq} holds for $\mu-$a.e. $R \in SO(d)$, with $\zeta_1$ close to $\Re(\nu)/|\Re(\nu)|$ and $\zeta_2$ close to being perpendicular to $\Re(\nu)/|\Re(\nu)|$, and reach a contradiction by choosing $R$ close to the identity. This completes the proof of Theorem  \ref{Thm1} in the cases (b) and (c) (and in fact, part of the case (a)).


\section{Proof of Theorem \ref{Thm1} - case (a): the outline of Christ-Shao and Foschi}

The goal of this section is to obtain the sharp inequality
\begin{equation}\label{Sec5_eq1}
\big\|\widehat{f \sigma}\big\|_{L^{4}(\R^{d})} \leq C(d,4,2)\, \|f\|_{L^{2}(\S^{d-1})}\,,
\end{equation}
for $3 \leq d \leq 7$, and to characterize its extremizers. A simple application of H\"{o}lder's inequality then gives the corresponding sharp inequalities for the cases $q >2$, as detailed in Section \ref{Holder reduction}. 

\medskip

In the case $d=3$, Foschi \cite{F} recently obtained the sharp inequality \eqref{Sec5_eq1} by combining previous techniques developed by Christ and Shao \cite{CS, CS2} with an insighful geometric identity intrinsic to this restriction problem. Foschi characterizes the real-valued extremizers of \eqref{Sec5_eq1} and completes the characterization of the complex-valued extremizers by invoking a result of Christ and Shao \cite[Theorem 1.2]{CS2}. Here we extend this method up to dimension $d=7$ to prove the sharp inequality \eqref{Sec5_eq1}, and characterize the complex-valued extremizers via a different path, using our Theorem \ref{Sec3_Char_ext} instead (note that, in principle, the result of \cite[Theorem 1.2]{CS2} is not available for dimensions $d >3$). 

\medskip

We keep the notation as close as possible to \cite{F} to facilitate some of the references. Lemmas \ref{Sec5_Lem8.1} - \ref{Sec5_Lem10} below are derived from the works of Christ and Shao \cite{CS, CS2} and Foschi \cite{F}. The novelty here is a careful discussion of the cases of equality.

\subsection{Reduction to nonnegative functions} Recall that by Plancherel's theorem we have 
\begin{align*}
\big\|\widehat{f \sigma}\big\|_{L^4(\R^d)}^2  = (2\pi)^{d/2}\, \big\|f\sigma * f\sigma\big\|_{L^2(\R^d)}.
\end{align*}
Our first lemma reduces matters to working with nonnegative functions.
\begin{lemma}\label{Sec5_Lem8.1} 
Let $f \in L^2(\S^{d-1})$. We have
\begin{equation}\label{Sec5_Lem8.1_Eq_1}
\big\|f\sigma * f\sigma\big\|_{L^2(\R^d)} \leq  \big\||f|\sigma * |f|\sigma\big\|_{L^2(\R^d)},
\end{equation}
with equality if and only if there is a measurable function $h:\ov{B(2)} \to \C$ such that 
\begin{equation}\label{Sec5_Lem8.1_Eq_2}
 f(\zeta_1)\, f(\zeta_2) = h(\zeta_1 + \zeta_2) \,\big|f(\zeta_1)\, f(\zeta_2)\big|
\end{equation}
for $\sigma^2-$a.e. $(\zeta_1, \zeta_2) \in (\S^{d-1})^2$. 
\end{lemma}
\begin{proof}
Recall from \eqref{Sec2_delta_form} that 
\begin{align*}
f\sigma * f\sigma(\xi) = \int_{(\S^{d-1})^2} f(\zeta_1)\, f(\zeta_2)\, \delta_d(\xi - \zeta_1 - \zeta_2)\, \dsigma(\zeta_1)\, \dsigma(\zeta_2),
\end{align*}
which implies that 
\begin{equation}\label{Sec5_red_pos_eq1}
|f\sigma * f\sigma(\xi) |\leq |f|\sigma * |f|\sigma(\xi) 
\end{equation}
for all $\xi \in \R^d$. This plainly gives \eqref{Sec5_Lem8.1_Eq_1}. 

\medskip

Assume we have equality in \eqref{Sec5_Lem8.1_Eq_1}. Then we must have equality in \eqref{Sec5_red_pos_eq1} for a.e. $\xi \in \R^d$. For each such $\xi \in \R^d$ there exists $h(\xi) \in \C$ such that 
\begin{equation}\label{Sec5_red_pos_eq2}
 f(\zeta_1)\, f(\zeta_2) = h(\xi) |f(\zeta_1)\, f(\zeta_2)|
\end{equation}
for $\Psi_{\xi}-$a.e. $(\zeta_1, \zeta_2) \in (\S^{d-1})^2$, where the singular measure $\Psi_{\xi}$ on $(\S^{d-1})^2$ is given by  
$$\d\Psi_{\xi}(\zeta_1, \zeta_2) =  \delta_d(\xi - \zeta_1 - \zeta_2)\, \dsigma(\zeta_1)\, \dsigma(\zeta_2).$$ 
By integrating with respect to $\Psi_{\xi}$ we find that 
\begin{equation}\label{Sec5_red_pos_eq3}
f\sigma * f\sigma(\xi) = h(\xi) \big(|f|\sigma * |f|\sigma(\xi)\big),
\end{equation}
and we see that $h$ is actually a measurable function. Arguing as in Section \ref{CE} we arrive at \eqref{Sec5_Lem8.1_Eq_2}.

\medskip

Conversely, if we have \eqref{Sec5_Lem8.1_Eq_2}, we may argue again as in Section \ref{CE} to conclude that for a.e. $\xi \in \R^d$ we have \eqref{Sec5_red_pos_eq2} for $\Psi_{\xi}-$a.e. $(\zeta_1, \zeta_2) \in (\S^{d-1})^2$. Then equality in \eqref{Sec5_red_pos_eq1} holds for a.e. $\xi \in \R^d$ and we have equality in \eqref{Sec5_Lem8.1_Eq_1}.
\end{proof}

By working with $|f|$ instead of $f$, {\it we may assume that we are dealing with nonnegative functions}.

\subsection{Reduction to even functions} Given a function $f: \S^{d-1} \to \R^+$ we define its {\it antipodal} $f_{\star}$ by 
\begin{equation*}
f_{\star}(\zeta) = f (-\zeta).
\end{equation*}
Using \eqref{Sec2_delta_form} we observe that 
\begin{align}\label{Sec5_Planc_star}
\begin{split}
\big\|f\sigma * f\sigma\big\|_{L^2(\R^d)}^2  & = \big\|f\sigma * f_{\star}\sigma\big\|_{L^2(\R^d)}^2\\
& = \int_{(\S^{d-1})^4} \!f(\zeta_1)\, f(-\zeta_2)\, f(\zeta_3)\, f(-\zeta_4)\,\delta_d (\zeta_1 + \zeta_2 + \zeta_3 + \zeta_4)\,\dsigma(\zeta_1)\,\dsigma(\zeta_2)\,\dsigma(\zeta_3)\,\dsigma(\zeta_4)\\ 
& = Q(f, f_{\star}, f, f_{\star}).
\end{split}
\end{align}
Here $Q$ is the quadrilinear form defined by
\begin{equation*}
Q(f_1, f_2, f_3, f_4) := \int_{(\S^{d-1})^4} f_1(\zeta_1)\, f_2(\zeta_2)\, f_3(\zeta_3)\, f_4(\zeta_4)\,\d\Sigma(\zeta),
\end{equation*}
where $\Sigma$ is the singular measure in $(\S^{d-1})^4$ given by
\begin{equation*}
\d\Sigma(\zeta) := \delta_d (\zeta_1 + \zeta_2 + \zeta_3 + \zeta_4)\,\dsigma(\zeta_1)\,\dsigma(\zeta_2)\,\dsigma(\zeta_3)\,\dsigma(\zeta_4),
\end{equation*}
for $\zeta = (\zeta_1, \zeta_2, \zeta_3, \zeta_4) \in (\S^{d-1})^4$. Note that $\Sigma$ is supported on the set $\{\zeta \in (\S^{d-1})^4;\ \zeta_1 + \zeta_2 + \zeta_3 + \zeta_4 = 0\}$.

\medskip

For $f: \S^{d-1} \to \R^+$, we define the {\it antipodal symmetrization} $f_{\sharp}$ by 
\begin{equation*}
f_{\sharp}(\zeta) = \sqrt{ \frac{f(\zeta)^2 + f(-\zeta)^2}{2}}.
\end{equation*}
Note that $\|f_{\sharp}\|_{L^2(\R^d)} = \|f\|_{L^2(\R^d)}$.
\begin{lemma}[cf. \cite{CS, F}] \label{Sec5_Lem8}
Let $d \geq 3$. If $f: \S^{d-1} \to \R^+$ belongs to $L^2(\S^{d-1})$ then
\begin{equation}\label{Sec5_ineq_Q}
Q(f, f_{\star}, f, f_{\star}) \leq Q(f_{\sharp}, f_{\sharp}, f_{\sharp}, f_{\sharp}),
\end{equation}
with equality if and only if $f= f_{\star} = f_{\sharp}$ {\rm (}$\sigma-$a.e. in $\S^{d-1}${\rm)}.
\end{lemma}
\begin{proof}
We follow \cite[Proposition 3.2]{F}. Observe first that 
\begin{equation}\label{Sec5_f_sharp_ineq}
f * f_{\star}(\xi) \leq f_{\sharp} * f_{\sharp}(\xi)
\end{equation}
for all $\xi \in \R^d$. In fact, we have
\begin{align}\label{Sec5_CS0}
\begin{split}
2 f * f_{\star}(\xi) &= f * f_{\star}(\xi) + f_{\star}* f(\xi)\\
& = \int_{(\S^{d-1})^2}\Big[ f(\zeta_1) f(-\zeta_2) + f(-\zeta_1)f(\zeta_2) \Big]\, \delta_d( \xi - \zeta_1 - \zeta_2)\,\dsigma(\zeta_1)\,\dsigma(\zeta_2).
\end{split}
\end{align}
By Cauchy-Schwarz inequality we have
\begin{equation}\label{Sec5_CS}
\Big[ f(\zeta_1) f(-\zeta_2) + f(-\zeta_1)f(\zeta_2) \Big]  \leq \sqrt{f(\zeta_1)^2 + f(-\zeta_1)^2} \ \sqrt{f(\zeta_2)^2 + f(-\zeta_2)^2} = 2 \,f_{\sharp}(\zeta_1)\,f_{\sharp}(\zeta_2).
\end{equation}
Plugging \eqref{Sec5_CS} into \eqref{Sec5_CS0} we obtain \eqref{Sec5_f_sharp_ineq}. Now observe that \eqref{Sec5_Planc_star} and \eqref{Sec5_f_sharp_ineq} plainly imply \eqref{Sec5_ineq_Q}.

\medskip

In order to have equality in \eqref{Sec5_ineq_Q}, we must have equality in \eqref{Sec5_f_sharp_ineq} for a.e. $\xi \in \R^d$. For each such $\xi \in \R^d$, the condition of equality in the Cauchy-Schwarz inequality \eqref{Sec5_CS} gives us that 
\begin{equation}\label{Sec5_cond_1}
f(\zeta_1)\, f(\zeta_2) = f(-\zeta_1)\, f(-\zeta_2)
\end{equation}
for $\Psi_{\xi}-$a.e. $(\zeta_1, \zeta_2) \in (\S^{d-1})^2$. Arguing as in Section \ref{CE}, this implies that \eqref{Sec5_cond_1} must hold for $\sigma^2-$a.e. $(\zeta_1, \zeta_2) \in (\S^{d-1})^2$. Let $\zeta_1 \in \S^{d-1}$ be such that \eqref{Sec5_cond_1} holds for $\sigma-$a.e. $\zeta_2 \in \S^{d-1}$. Then we can integrate over $\S^{d-1}$ with respect to the variable $\zeta_2$ to obtain (provided $f$ is nonzero, otherwise the result is trivial)
\begin{equation*}
f(\zeta_1) = f(-\zeta_1).
\end{equation*}
This shows that $f= f_{\star} = f_{\sharp}$ {\rm (}$\sigma-$a.e. in $\S^{d-1}${\rm)}.
\end{proof}

From now on {\it we may assume additionally that $f = f_{\sharp}$}.

\subsection{The key geometric identity} The heart of Foschi's proof lies in the following simple geometric identity.

\begin{lemma}{\rm (cf. \cite[Lemma 4.2]{F})}
Let $(\zeta_1, \zeta_2, \zeta_3, \zeta_4) \in (\S^{d-1})^4$ be such that
\begin{equation*}
\zeta_1 + \zeta_2 + \zeta_3 + \zeta_4 = 0
\end{equation*}
{\rm (}i.e. $(\zeta_1, \zeta_2, \zeta_3, \zeta_4) $ lies in the support of the measure $\Sigma${\rm )}. Then
\begin{equation}\label{Sec5_magical_identity}
|\zeta_1 + \zeta_2|\,|\zeta_3 + \zeta_4| + |\zeta_1 + \zeta_3|\,|\zeta_2 + \zeta_4| + |\zeta_1 + \zeta_4|\,|\zeta_2 + \zeta_3| = 4.
\end{equation}
\end{lemma}
The kernel in our Corollary \ref{Cor3} is too singular to allow us to draw any sharp global conclusions about the adjoint restriction inequality \eqref{Sec5_eq1}. To overcome this barrier we use the identity \eqref{Sec5_magical_identity} and the symmetries of $Q$ in order to write
\begin{equation}\label{Sec5_Id_Q}
Q(f,f,f,f) = \frac{3}{4}  \int_{(\S^{d-1})^4} f(\zeta_1)\, f(\zeta_2)\, |\zeta_1 + \zeta_2|\, f(\zeta_3)\, f(\zeta_4)\,|\zeta_3 + \zeta_4|\,\d\Sigma(\zeta).
\end{equation}
This allows us to prove the following lemma.
\begin{lemma}\label{Sec5_Lem10}
Let $d \geq 3$. If $f:\S^{d-1} \to \R^{+}$ is an even function in $L^2(\S^{d-1})$ then
\begin{equation*}
Q(f,f,f,f) \leq 2^{-d+3}\,\omega_{d-2}\, \frac{3}{4} \,\int_{(\S^{d-1})^2} f(\zeta_1)^2\, f(\zeta_2)^2\, |\zeta_1 + \zeta_2|\,\big(4-|\zeta_1 + \zeta_2|^2 \big)^{\frac{d-3}{2}}\,\dsigma(\zeta_1) \, \dsigma (\zeta_2),
\end{equation*}
with equality if and only if $f$ is a constant function.
\end{lemma}

\begin{proof} We use the Cauchy-Schwarz inequality in \eqref{Sec5_Id_Q}  (with respect to the measure $\Sigma$), together with Lemma \ref{biconvsigma} to get
\begin{align*}
\begin{split}
Q(f,f,f,f) &\leq \frac{3}{4} \left(\int_{(\S^{d-1})^4} f(\zeta_1)^2\, f(\zeta_2)^2\, |\zeta_1 + \zeta_2|^2\,\d\Sigma(\zeta)\right)
^{1/2} \left(\int_{(\S^{d-1})^4} f(\zeta_3)^2\, f(\zeta_4)^2\, |\zeta_3 + \zeta_4|^2\,\d\Sigma(\zeta)\right)
^{1/2}\\
& = \frac{3}{4} \int_{(\S^{d-1})^4} f(\zeta_1)^2\, f(\zeta_2)^2\, |\zeta_1 + \zeta_2|^2\,\d\Sigma(\zeta)\\
& = 2^{-d+3}\,\omega_{d-2}\, \frac{3}{4} \,\int_{(\S^{d-1})^2} f(\zeta_1)^2\, f(\zeta_2)^2\, |\zeta_1 + \zeta_2|\,\big(4-|\zeta_1 + \zeta_2|^2 \big)^{\frac{d-3}{2}}\,\dsigma(\zeta_1) \, \dsigma (\zeta_2).
\end{split}
\end{align*}
In order to have equality we must have
\begin{equation*}
f(\zeta_1)\, f(\zeta_2)\, |\zeta_1 + \zeta_2| =  c \, f(\zeta_3)\, f(\zeta_4)\,|\zeta_3 + \zeta_4|
\end{equation*}
for some $c \in \R$ and $\Sigma-$a.e. $(\zeta_1, \zeta_2, \zeta_3, \zeta_4) \in (\S^{d-1})^4$. Integrating both sides with respect to $\d\Sigma(\zeta_1, \zeta_2, \zeta_3, \zeta_4)$ gives us that $c = 1$ and thus 
\begin{equation}\label{Sec5_Lem10_Eq_cond}
f(\zeta_1)\, f(\zeta_2)\, |\zeta_1 + \zeta_2| =  f(\zeta_3)\, f(\zeta_4)\,|\zeta_3 + \zeta_4|
\end{equation}
for $\Sigma-$a.e. $(\zeta_1, \zeta_2, \zeta_3, \zeta_4) \in (\S^{d-1})^4$. Let 
\begin{equation}\label{Sec5_Def_E}
E = \big\{(\zeta_1, \zeta_2, \zeta_3, \zeta_4) \in (\S^{d-1})^4;\ \eqref{Sec5_Lem10_Eq_cond}\ {\rm does \ not\ hold}\}.
\end{equation}
We find that 
\begin{align*}
0 &= \int_{(\S^{d-1})^4} \chi_E(\zeta_1, \zeta_2, \zeta_3, \zeta_4)\, \d\Sigma(\zeta_1, \zeta_2, \zeta_3, \zeta_4) \\
& =  \int_{(\S^{d-1})^2} \left( \int_{(\S^{d-1})^2} \chi_E(\zeta_1, \zeta_2, -\zeta_3, -\zeta_4) \,\delta_d(\zeta_1 + \zeta_2 - \zeta_3 - \zeta_4)\,\dsigma(\zeta_3)\, \dsigma(\zeta_4)\right) \dsigma(\zeta_1)\,\dsigma(\zeta_2).
\end{align*}
Thus, for $\sigma^2-$a.e. $(\zeta_1, \zeta_2) \in (\S^{d-1})^2$, we have that (recall that $f$ is even)
\begin{equation}\label{Sec5_Lem10_Eq_cond_2}
f(\zeta_1)\, f(\zeta_2)\, |\zeta_1 + \zeta_2| =  f(\zeta_3)\, f(\zeta_4)\,|\zeta_3 + \zeta_4|
\end{equation}
for $\Psi_{\zeta_1 + \zeta_2}-$a.e. $(\zeta_3, \zeta_4) \in (\S^{d-1})^2$. For such a $(\zeta_1, \zeta_2) \in (\S^{d-1})^2$, with $\zeta_1 + \zeta_2 \in B(2)\setminus \{0\}$ we have
\begin{equation}\label{Sec5_Lem10_Eq_cond_3}
f(\zeta_1)\, f(\zeta_2) =  f(\zeta_3)\, f(\zeta_4)
\end{equation}
for $\Psi_{\zeta_1 + \zeta_2}-$a.e. $(\zeta_3, \zeta_4) \in (\S^{d-1})^2$, and if we average the right-hand side of \eqref{Sec5_Lem10_Eq_cond_3} with respect to $\d\Psi_{\zeta_1 + \zeta_2}(\zeta_3,\zeta_4)$ we arrive at
\begin{align*}
f(\zeta_1)\, f(\zeta_2) &= \frac{ \int_{(\S^{d-1})^2} f(\zeta_3)\, f(\zeta_4) \,\d\Psi_{\zeta_1 + \zeta_2}(\zeta_3,\zeta_4)}{ \int_{(\S^{d-1})^2}\d\Psi_{\zeta_1 + \zeta_2}(\zeta_3,\zeta_4)}=  \frac{f\sigma*f\sigma(\zeta_1 + \zeta_2)}{\sigma^{(2)}(\zeta_1 + \zeta_2)}=:h(\zeta_1 + \zeta_2).
\end{align*}

\medskip

We now use Theorem \ref{Sec3_Char_ext} to conclude that $f(\zeta) = c\, e^{\nu\cdot \zeta}$ for some $c \in \C$ and $\nu \in \C^d$. If $c = 0$ we are done. If $c \neq 0$, since $f$ is real-valued we must have $\Im(\nu) =0$, and since $f$ is even we must have $\Re(\nu) = 0$. Then, since $f$ is nonnegative, we must have $c>0$ and $f(\zeta) = c$.

\medskip

Conversely, it is clear that any (nonnegative) constant function verifies the desired equality. This concludes the proof.
\end{proof}

\subsection{Proof of Theorem \ref{Thm1} - case (a)} We consider the quadratic form 
\begin{equation*}
H_d(g) := \int_{(\S^{d-1})^2} \ov{g(\zeta_1)}\, g(\zeta_2)\, |\zeta_1 - \zeta_2|\,\big(4-|\zeta_1 -\zeta_2|^2 \big)^{\frac{d-3}{2}}\,\dsigma(\zeta_1) \, \dsigma (\zeta_2).
\end{equation*}
This is a real-valued and continuous functional on $L^1(\S^{d-1})$. In fact, it is not hard to see that 
\begin{equation}\label{Sec5_Cont_H_d}
|H_d(g_1) - H_d(g_2)| \leq 2^{d-2}\,  \Big(\|g_1\|_{L^1(\S^{d-1})} + \|g_2\|_{L^1(\S^{d-1})}\Big)\, \|g_1 - g_2\|_{L^1(\S^{d-1})}.
\end{equation}
The next lemma is the last piece of information needed for our sharp inequality.

\begin{lemma}\label{Sec5_Lem11}
Let $3 \leq d \leq 7$. Let $g \in L^1(\S^{d-1})$ be an even function and write
\begin{equation*}
\mu = \frac{1}{\omega_{d-1}} \int_{\S^{d-1}} g(\zeta)\,\dsigma(\zeta)
\end{equation*}
for the mean value of $g$ over the sphere $\S^{d-1}$. Then 
\begin{equation*}
H_d(g) \leq H_d(\mu {\bf 1}) = |\mu|^2 H_d({\bf 1}),
\end{equation*}
with equality if and only if $g$ is a constant function.
\end{lemma}

Assume for a moment that we have proved this lemma and let us conclude the proof of Theorem \ref{Thm1}.

\begin{proof}[Proof of Theorem \ref{Thm1} - case {\rm (a)}] Putting together our chain of inequalities (Lemmas \ref{Sec5_Lem8.1}, \ref{Sec5_Lem8}, \ref{Sec5_Lem10} and \ref{Sec5_Lem11}) we get
\begin{align}\label{Sec5_chain}
\begin{split}
\big\|\widehat{f \sigma}\big\|_{L^4(\R^d)}^4&  \leq (2\pi)^d\, Q(|f|, |f|_{\star}, |f|, |f|_{\star}) \\
&  \leq  (2\pi)^d\, Q(|f|_{\sharp}, |f|_{\sharp}, |f|_{\sharp}, |f|_{\sharp})\\
& \leq (2\pi)^d\, 2^{-d+3}\,\omega_{d-2}\, \frac{3}{4} \,H_d(|f|_{\sharp}^2)\\
& \leq   \frac34\,(2\pi)^d\, 2^{-d+3}\,\frac{\omega_{d-2}}{\omega_{d-1}^2}\, H_d({\bf 1}) \,\|f\|_{L^2(\S^{d-1})}^4.
\end{split}
\end{align}
This inequality is sharp since $f = {\bf 1}$ verifies the equalities in all the steps. 

\medskip

If $f \in L^2(\S^{d-1})$ is a complex-valued extremizer of \eqref{Sec5_chain}, by Lemma \ref{Sec5_Lem10} (or Lemma \ref{Sec5_Lem11}) we must have $|f|_{\sharp} = \gamma\,{\bf 1}$, where $\gamma >0$ is a constant. By Lemma \ref{Sec5_Lem8} we must have $|f| = \gamma\,{\bf 1}$. By Lemma \ref{Sec5_Lem8.1} there is a measurable function $h:\ov{B(2)} \to \C$ such that 
\begin{equation*}
 f(\zeta_1)\, f(\zeta_2) = \gamma^2\, h(\zeta_1 + \zeta_2) 
\end{equation*}
for $\sigma^2-$a.e. $(\zeta_1, \zeta_2) \in (\S^{d-1})^2$. By Theorem \ref{Sec3_Char_ext} there exist $c \in \C\setminus \{0\}$ and $\nu \in \C^d$ such that 
$$f(\zeta) = c\,e^{\nu \cdot \zeta}$$
for $\sigma-$a.e. $\zeta \in \S^{d-1}$. Since $|f|$ is constant, we must have $\Re(\nu) = 0$ and $|c| = \gamma$. 

\medskip

Conversely, it is clear that the functions given by \eqref{CV_ext} verify the chain of equalities in \eqref{Sec5_chain}. This concludes the proof.
\end{proof}

\subsection{Spectral decomposition - Proof of Lemma \ref{Sec5_Lem11}} The case $d=3$ was proved by Foschi \cite[Theorem 5.1]{F}. Here we extend his method to dimensions $d = 4,5,6,7$.

\subsubsection{Funk-Hecke formula and Gegenbauer polynomials} We start by proving Lemma \ref{Sec5_Lem11} for even functions $g$ in the subspace $L^2(\S^{d-1}) \subset L^1(\S^{d-1})$ (the general statement for even functions in $L^1(\S^{d-1})$ will follow by a density argument). In this case we may decompose $g$ as a sum 
\begin{equation}\label{Sec5_Dec_Spherical_Harmonics}
g = \sum_{k \geq 0} Y_k,
\end{equation}
where $Y_k$ is a spherical harmonic of degree $k$ (see \cite[Chapter IV]{SW}). Since $g$ is even, we must have $Y_{2\ell+1} = 0$ in \eqref{Sec5_Dec_Spherical_Harmonics} for all $\ell\geq 0$. Note also that $Y_0 = \mu \,{\bf 1}$, where $\mu$ is the mean value of $g$ in $\S^{d-1}$. Let $$g_N = \sum_{k \geq 0}^N Y_k.$$ 
Since $g_N \to g$ in $L^2(\S^{d-1})$ as $N \to \infty$, we have that $g_N \to g$ in $L^1(\S^{d-1})$ as $N \to \infty$ and thus, by \eqref{Sec5_Cont_H_d}, $H_d(g_N) \to H_d(g)$ as $N \to \infty$. Therefore
\begin{align}\label{Sec6_H_d_decomposition}
\begin{split}
H_d(g) & = \lim_{N \to \infty} \sum_{j,k=0}^N \int_{\S^{d-1}} \int_{\S^{d-1}} \ov{Y_j(\zeta_1)}\, Y_k(\zeta_2)\, |\zeta_1 - \zeta_2|\,\big(4-|\zeta_1 -\zeta_2|^2 \big)^{\frac{d-3}{2}}\,\dsigma(\zeta_1) \, \dsigma (\zeta_2)\\
& = \lim_{N \to \infty} \sum_{j,k=0}^N \int_{\S^{d-1}} \ov{Y_j(\zeta_1)}\,\left( \int_{\S^{d-1}}  Y_k(\zeta_2)\, \phi_d(\zeta_1 \cdot \zeta_2)\,\dsigma(\zeta_2) \right) \dsigma (\zeta_1),
\end{split}
\end{align}
where 
\begin{equation}\label{Sec6_def_phi_d}
\phi_d(t) = 2^{\frac{d-2}{2}}\, (1 - t)^{\frac12}\, (1+t)^{\frac{d-3}{2}}.
\end{equation} 
The inner integral above may be evaluated via the {\it Funk-Hecke formula} \cite[Theorem 1.2.9]{DX} 
\begin{equation}\label{Sec6_FH}
\int_{\S^{d-1}}  Y_k(\zeta_2)\, \phi_d(\zeta_1 \cdot \zeta_2)\,\dsigma(\zeta_2) = \Lambda_k(\phi_d) \, Y_k(\zeta_1),
\end{equation}
with the constant $\Lambda_k(\phi_d)$ given by
\begin{equation}\label{Sec6_Def_Lambda_k}
\Lambda_k(\phi_d) = \omega_{d-2}\int_{-1}^{1} \frac{C^{\frac{d-2}{2}}_k(t)}{C^{\frac{d-2}{2}}_k(1)}\, \phi_d(t)\, (1-t^2)^{\frac{d-3}{2}}\, \dt,
\end{equation}
where $t \mapsto C^{\alpha}_k(t)$, for $\alpha >0$, are the {\it Gegenbauer polynomials} ({\it or ultraspherical polynomials}) defined in terms of the generating function 
\begin{equation}\label{Sec6_Def_Gegenbauer}
(1 - 2rt + r^2)^{-\alpha} = \sum_{k=0}^{\infty} C^{\alpha}_k(t)\, r^k.
\end{equation}
For bounded $t$, the left-hand side of \eqref{Sec6_Def_Gegenbauer} is an analytic function of $r$ (for small $r$) and the right-hand side of \eqref{Sec6_Def_Gegenbauer} is the corresponding power series expansion. Note that $C^{\alpha}_k(t)$ has degree $k$. The Gegenbauer polynomials $C^{\alpha}_k(t)$ are orthogonal in the interval $[-1,1]$ with respect to the measure $(1- t^2)^{\alpha - \frac12}\,\dt$. If we plug \eqref{Sec6_FH} back into \eqref{Sec6_H_d_decomposition} and use the orthogonality properties of the spherical harmonics we arrive at
\begin{equation}\label{Sec6_decom_H_g_Sph_Harm_Final}
H_d(g) = \sum_{k=0}^{\infty} \Lambda_k(\phi_d) \, \|Y_k\|^2_{L^2(\S^{d-1})}.
\end{equation}
Our goal here is prove the following result.

\begin{lemma}\label{Lem13}.
\begin{itemize}
\item[(i)] For $d=3$ we have $\Lambda_0(\phi_3) >0$ and $\Lambda_k(\phi_3) <0$ for all $k \geq1$.

\item[(ii)] For $d=4$ we have $\Lambda_0(\phi_4) >0$,  $\Lambda_{2k +1}(\phi_4) =0$ for all $k \geq0$, and $\Lambda_{2k}(\phi_4) <0$ for all $k \geq 1$.

\item[(iii)] For $d = 5,6,7$ we have $\Lambda_0(\phi_d),  \Lambda_1(\phi_d)>0$ and $\Lambda_k(\phi_d) <0$ for all $k \geq 2$.

\end{itemize}
\end{lemma}

\noindent{\sc Remark:} For $d \geq 8$ we start to observe that $\Lambda_2(\phi_d)>0$. This is the basic reason why the method presented here only works (as it is) for dimensions up to $7$.

\medskip

Assuming Lemma \ref{Lem13} let us conclude the proof of Lemma \ref{Sec5_Lem11}.

\begin{proof}[Conclusion of the proof of Lemma \ref{Sec5_Lem11}] Using Lemma \ref{Lem13} in \eqref{Sec6_decom_H_g_Sph_Harm_Final}, and the fact that $g \in L^2(\S^{d-1})$ is an even function, we find
\begin{align*}
H_d(g)  = \sum_{k=0}^{\infty} \Lambda_k(\phi_d) \, \|Y_k\|^2_{L^2(\S^{d-1})}&  \leq  \Lambda_0(\phi_d) \, \|Y_0\|^2_{L^2(\S^{d-1})} + \Lambda_1(\phi_d) \, \|Y_1\|^2_{L^2(\S^{d-1})} \\
& = \Lambda_0(\phi_d) \, \|Y_0\|^2_{L^2(\S^{d-1})}\\
& = H_d(\mu {\bf 1}) \\
& = |\mu|^2 H_d({\bf 1}).
\end{align*}
Equality occurs if and only if $Y_k \equiv0$ for all $k \geq 2$, which means that $g = Y_0$ is a constant function.

\medskip

Now let $h \in L^1(\S^{d-1})$ be an even function. For each $N >0$ consider the truncation
\begin{equation*}
h_N(\zeta) = \left\{
\begin{array}{cc}
h(\zeta) \ & {\rm if}\ |h(\zeta)| \leq N \vspace{0.15cm}\\

N \frac{h(\zeta)}{|h(\zeta)|} \ & {\rm if}\ |h(\zeta)| > N.
\end{array}
\right.
\end{equation*}
Note that each $h_N$ is an even function in $L^2(\S^{d-1})$ and $h_N \to h$ in $ L^1(\S^{d-1})$ as $N \to \infty$. If $\mu _N = \frac{1}{\omega_{d-1}} \int_{\S^{d-1}} h_N(\zeta)\,\dsigma(\zeta)$ and $\mu = \frac{1}{\omega_{d-1}} \int_{\S^{d-1}} h(\zeta)\,\dsigma(\zeta)$, then $\mu_N \to \mu$ and $H_d(h_N) \to H_d(h)$ as $N \to \infty$. Consider an orthonormal basis of spherical harmonics $\{Z_{kj}\}$, $k \geq 0$, $j = 1,2, \ldots,D(d,k)$, where each harmonic $Z_{kj}$ has degree $k$ and $D(d,k)$ is the dimension \footnote{Explicitly, $D(d,k) = \binom{d + k -1}{d-1} - \binom{d + k -3}{d-1}$. For details see \cite[Chapter IV]{SW}.} of the vector space of spherical harmonics of degree $k$ on $\S^{d-1}$, and write 
\begin{equation*}
h_N = \sum_{k,j} \langle h_N, Z_{kj}\rangle \, Z_{kj}.
\end{equation*}
By the $L^2-$argument we have
\begin{equation}\label{Sec6_case_L_1}
H_d(h_N) = \sum_{k=0}^{\infty} \Lambda_k(\phi_d) \left( \sum_{j=1}^{D(d,k)} |\langle h_N, Z_{kj}\rangle|^2  \right) \leq  |\mu_N|^2 H_d({\bf 1}).
\end{equation}
Passing the limit as $N\to \infty$ we get 
\begin{equation}\label{Sec6_case_L_1_eq_2}
H_d(h) \leq  |\mu|^2 H_d({\bf 1}),
\end{equation}
which is our desired inequality. We already know that $\langle h, Z_{kj}\rangle = 0$ if $k$ is odd. If $\langle h, Z_{kj}\rangle \neq 0$ for some even $k \geq 2$  and some $j$, since $\langle h_N, Z_{kj}\rangle \to \langle h, Z_{kj}\rangle$ as $N \to \infty$, by Lemma \ref{Lem13} we would have a strict inequality in \eqref{Sec6_case_L_1} that would propagate to the limit \eqref{Sec6_case_L_1_eq_2}. Therefore, in order to have the equality in \eqref{Sec6_case_L_1_eq_2} we must have $\langle h, Z_{kj}\rangle = 0$ for all $k \geq 1$ and all $j$. This implies that $h$ must be constant.
\end{proof}

\subsubsection{Computing the $\Lambda_k(\phi_d)$} We are left with the final task to prove Lemma \ref{Lem13}. In this section we accomplish this by explicitly computing the values of $\Lambda_k(\phi_d)$ via a recursive argument. To simplify the notation let us consider the {\it Legendre polynomials} defined by
\begin{equation*}
P_k(t) := C_k^{1/2}(t).
\end{equation*} 
The next proposition lays the ground for our recursions.
\begin{proposition}\label{Prop14}.
\begin{itemize}
\item[(i)] For $\alpha >0$ and $k \geq 0$ we have
\begin{equation}\label{Sec6_Prop14_eq1}
(2k + 2\alpha)\, t\, C^{\alpha}_k(t) = (k+1) \,C^{\alpha}_{k+1}(t) + (k + 2\alpha -1)\,C^{\alpha}_{k-1}(t).
\end{equation}
Note: We set $C_{-1}^{\alpha}(t) = 0$.

\smallskip

\item[(ii)] Let $k \geq 0$. The Legendre polynomials $P_k$ verify
\begin{align}
\gamma_k:=&\int_{-1}^1 (2-2t)^{1/2}\,P_k'(t)\,\dt=2\,(-1)^{k+1}+\frac{{2}}{2k+1}\,;\label{Sec6_Prop14_eq3}\\
\delta_k:=&\int_{-1}^1 (2-2t)^{3/2}\,P_k''(t)\,\dt= 8\,(-1)^{k}{k+1 \choose 2}+3\gamma_k. \label{Sec6_Prop14_eq4}
\end{align}
\end{itemize}
\end{proposition}

\begin{proof} {\it Part} (i). Differentiating \eqref{Sec6_Def_Gegenbauer} with respect to the variable $r$ yields
\begin{equation}\label{Sec6_eq_derivated}
2\alpha\,(t-r)\, \sum_{k=0}^{\infty} C^{\alpha}_k(t)\, r^k = (1 - 2rt + r^2)\,\sum_{k=0}^{\infty} k\,C^{\alpha}_k(t)\, r^{k-1}.
\end{equation}
Comparing the coefficients of $r^k$ on both sides of \eqref{Sec6_eq_derivated} we obtain \eqref{Sec6_Prop14_eq1}.

\medskip

\noindent{\it Part} (ii). To establish \eqref{Sec6_Prop14_eq3} we use integration by parts:

\begin{equation}\label{ibp}
\int_{-1}^1 (2-2t)^{1/2}\,P_k'(t)\,\dt=-2P_k(-1)+\int_{-1}^1\frac{P_k(t)}{(2-2t)^{1/2}}\,\dt.
\end{equation}
By evaluating 
\begin{equation}\label{Sec6_Def_Legendre}
(1 - 2rt + r^2)^{-1/2} = \sum_{k=0}^{\infty} P_k(t)\, r^k
\end{equation}
at $t=-1$ we find that $P_k(-1)=(-1)^k$. The value of the integral on the right-hand side of \eqref{ibp} was evaluated in the proof of \cite[Lemma 5.4]{F} and equals $\frac{2}{2k+1}$.

\medskip

As for identity \eqref{Sec6_Prop14_eq4}, integrate by parts once again:
\begin{equation*}
\int_{-1}^1 (2-2t)^{3/2}\,P_k''(t)\,\dt=-8P_k'(-1)+3\int_{-1}^1 (2-2t)^{1/2} \,P_k'(t)\,\dt.
\end{equation*}
Differentiating both sides of \eqref{Sec6_Def_Legendre} with respect to $t$ and then setting $t=-1$ yields 
\begin{equation*}
\frac{r}{(1+r)^3}=\sum_{k\geq 0} P_k'(-1)\,r^k.
\end{equation*}
Hence $P_k'(-1)=(-1)^{k+1}{k+1 \choose 2}$.
The result follows from this and \eqref{Sec6_Prop14_eq3}.
\end{proof}

We are now able to proceed to the proof of Lemma \ref{Lem13}.

\begin{proof}[Proof of Lemma \ref{Lem13}].

\medskip

\noindent {\it Step 1: Case} $d=3$. This was done in \cite[Lemma 5.4]{F}.

\medskip

\noindent {\it Step 2: Case} $d=4$. \noindent Throughout this proof let us rename
\begin{equation*}
Q_k(t) := C_k^1(t).
\end{equation*}
From \eqref{Sec6_Def_Gegenbauer} we find that $Q_k(1) = k+1$ for all $k \geq 0$, and from \eqref{Sec6_Prop14_eq1} we get
\begin{equation*}
2t\, Q_k(t)=Q_{k+1}(t)+Q_{k-1}(t).
\end{equation*}
Using this recursively we get, for $k \geq 2$, 
\begin{align*}
4\,t^2 Q_k(t)& =2t\,Q_{k+1}(t)+2t\,Q_{k-1}(t)\\
& =Q_{k+2}(t)+2\,Q_k(t)+Q_{k-2}(t),
\end{align*}
and thus
\begin{equation}\label{recQk}
(4t^2-2)\,Q_k(t)=Q_{k+2}(t)+Q_{k-2}(t).
\end{equation}
We are now ready to compute the coefficients $\Lambda_k(\phi_4)$:
\begin{align*}
\Lambda_k(\phi_4)& =\omega_2\int_{-1}^1\frac{Q_k(t)}{Q_k(1)}\,\phi_4(t)\,(1-t^2)^{1/2}\,\dt=\frac{\omega_2}{2(k+1)}\int_{-1}^1Q_k(t)\,\big(2-(4t^2-2)\big)\,\dt.
\end{align*}
Setting $\tau_k:=\int_{-1}^1Q_k(t)\,\dt$ and recalling \eqref{recQk}, we have 
\begin{equation}\label{lambdabeta}
\frac{2(k+1)}{\omega_2}\Lambda_k(\phi_4)=2\tau_k-\tau_{k-2}-\tau_{k+2}\,; \ \ \ \  (k\geq 2).
\end{equation}
The sequence of moments $\{\tau_k\}_{k \geq 0}$ can be computed explicitly. In fact, we claim that $\tau_{2j}=\frac{2}{2j+1}$ and $\tau_{2j+1}=0$. To verify this, recall by \eqref{Sec6_Def_Gegenbauer} that
\begin{equation*}
\sum_{k\geq 0} Q_k(t)\,r^k=\left(\sum_{k\geq 0} P_k(t)\,r^k\right)^2,
\end{equation*}
and so
\begin{equation*}
Q_k(t)=\sum_{\ell=0}^k P_{\ell}(t)\,P_{k-{\ell}}(t).
\end{equation*}
It follows that
\begin{equation*}
\tau_k=\int_{-1}^1 Q_k(t)\,\dt=\sum_{\ell=0}^k \int_{-1}^1 P_{\ell}(t)\,P_{k-\ell}(t) \,\dt.
\end{equation*}
By the orthogonality properties of Legendre polynomials, we find that $\tau_{2j+1}=0$. If $k=2j$ is even, then\footnote{This is also quoted in \cite[Lemma 5.4]{F}.}
\begin{equation*}
\tau_{2j}=\int_{-1}^1 P_j(t)^2\,\dt=\frac{2}{2j+1},
\end{equation*}
as claimed. Plugging this into \eqref{lambdabeta}, we immediately check that $\Lambda_{2j+1}(\phi_4)=0$ for every $j\geq 0$, and that
\begin{equation*}
\Lambda_{2j}(\phi_4)=\frac{\omega_2}{2(2j+1)}\left(\frac{4}{2j+1}-\frac{2}{2j-1}-\frac{2}{2j+3}\right)<0,
\end{equation*}
for every $j\geq 1$. If $j =0$ then
\begin{equation*}
\Lambda_0(\phi_4)=\omega_2\int_{-1}^1\phi_4(t)\,(1-t^2)^{1/2}\,\dt=2\,\omega_2\int_{-1}^1 (1-t^2)\,\dt=\frac{8\,\omega_2}{3}=\frac{32\pi}{3}>0.
\end{equation*}

\medskip

\noindent {\it Step 3: Case} $d=5$. In order to simplify the notation, we start again by relabeling
\begin{equation*}
R_k(t) := C^{3/2}_k(t).
\end{equation*}
The definition \eqref{Sec6_Def_Gegenbauer} gives us
\begin{equation}\label{defRk} 
(1 - 2rt + r^2)^{-3/2} = \sum_{k=0}^{\infty} R_k(t)\, r^k.
\end{equation}
Differentiating \eqref{Sec6_Def_Legendre} with respect to the variable $t$ yields
\begin{equation*}
r(1-2rt+r^2)^{-3/2}=\sum_{k\geq 0}P_k'(t)\,r^k.
\end{equation*}
Comparing the two last displays, one concludes that 
\begin{equation*}
R_k(t)=P_{k+1}'(t)
\end{equation*}
for every $k \geq 0$. From \eqref{defRk} we have $R_k(1)={{k+2}\choose{2}}$. It follows from \eqref{Sec6_def_phi_d} and \eqref{Sec6_Def_Lambda_k} that
\begin{align*}
\frac{1}{2\omega_3}{{k+2}\choose{2}}\Lambda_k(\phi_5)&=\int_{-1}^1(2-2t)^{1/2}\,R_k(t)\,(1+t)\,(1-t^2)\,\dt\\
&=\int_{-1}^1(2-2t)^{1/2}\,P_{k+1}'(t)\,(1+t-t^2-t^3)\,\dt\\
&= \gamma_{k+1}^{(0)} + \gamma_{k+1}^{(1)} - \gamma_{k+1}^{(2)} - \gamma_{k+1}^{(3)},
\end{align*}
where
\begin{equation*}
\gamma_{k}^{(j)} = \int_{-1}^1(2-2t)^{1/2}\,P_{k}'(t)\, t^j\,\dt.
\end{equation*}
Note that $\gamma_{0}^{(j)} = 0$ for any $j \geq 0$. From \eqref{Sec6_Prop14_eq3} we know the exact value of $\gamma_{k}^{(0)}$. From \eqref{Sec6_Prop14_eq1} we have (recall that we have set $R_{-1}(t) =0$)
\begin{equation*}
tR_k(t)=\frac{k+1}{2k+3}R_{k+1}(t)+\frac{k+2}{2k+3} R_{k-1}(t),
\end{equation*}
which plainly gives
\begin{align*}
\gamma_{k+1}^{(j+1)} & =  \frac{k+1}{2k+3} \gamma_{k+2}^{(j)} + \frac{k+2}{2k+3} \gamma_{k}^{(j)} 
\end{align*}
for all $k \geq 0$ and $j \geq 0$. The previous recursion tells us that we can explicitly compute the values of all $\gamma_{k}^{(j)}$ by just knowing the values of $\gamma_{0}^{(j)}$ for all $j \geq 0$ and the values of $\gamma_{k}^{(0)}$ for all $k \geq 0$. This computation leads us to
\begin{align*}
\frac{1}{2\omega_3}{{k+2}\choose{2}}\Lambda_k(\phi_5) & = \gamma_{k+1}^{(0)} + \gamma_{k+1}^{(1)} - \gamma_{k+1}^{(2)} - \gamma_{k+1}^{(3)}\\
& = \frac{768 ( k+1) ( k+2) (3 - 3 k - k^2)}{( 2 k-3) ( 2 k-1) ( 2 k+1) ( 2 k+3) ( 2 k+5) ( 2 k+7) (2 k+9)},
\end{align*}
and it follows that $\Lambda_0(\phi_5) , \Lambda_1(\phi_5) >0$ and $\Lambda_k(\phi_5) <0$ if $k\geq 2$, as claimed.

\medskip

\noindent {\it Step 4: Case} $d=6$. We set
\begin{equation*}
S_k(t) := C_k^2(t).
\end{equation*}
The definition \eqref{Sec6_Def_Gegenbauer} gives us
\begin{equation}\label{Sec6_def_S_k}
(1 - 2rt + r^2)^{-2} = \sum_{k=0}^{\infty} S_k(t)\, r^k\,,
\end{equation}
and it follows that $S_k(1) = {{k+3}\choose{3}}$. By differentiating \eqref{Sec6_Def_Gegenbauer}, in the case $\alpha = 1$, with respect to the variable $t$ and comparing coefficients with \eqref{Sec6_def_S_k} we find
\begin{equation}\label{Sec6_relation_S_Q}
S_k(t) = \tfrac12 \,Q_{k+1}'(t)
\end{equation}
for all $k \geq 0$. From \eqref{Sec6_def_phi_d} and \eqref{Sec6_Def_Lambda_k} it follows that
\begin{align*}
\frac{1}{4\,\omega_4}{{k+3}\choose{3}}\Lambda_k(\phi_6) & = \int_{-1}^1 S_k(t)\,(1 + t - 2 t^2 - 2 t^3 + t^4 + t^5)\,\dt\\
& = \epsilon_k^{(0)} + \epsilon_k^{(1)} - 2\epsilon_k^{(2)} - 2\epsilon_k^{(3)} + \epsilon_k^{(4)} + \epsilon_k^{(5)},
\end{align*}
where 
\begin{equation*}
\epsilon_k^{(j)} = \int_{-1}^1 S_k(t)\,t^j\,\dt.
\end{equation*}
Since Gegenbauer polynomials of odd degree are odd, it follows that $\epsilon_{2\ell+1}^{(0)}=0$ for every $\ell\geq 0$. On the other hand, using \eqref{Sec6_relation_S_Q}, we have
\begin{equation*}
\epsilon_{2\ell}^{(0)} = \int_{-1}^1 S_{2\ell}(t)\,\dt = \frac12 \big(Q_{2\ell+1}(1) - Q_{2\ell+1}(-1) \big) = 2(\ell + 1)
\end{equation*}
for every $\ell \geq 0$. From \eqref{Sec6_Prop14_eq1} we have (recall that we have set $S_{-1}(t) =0$)
\begin{equation*}
tS_k(t)=\frac{k+1}{2k+4}S_{k+1}(t)+\frac{k+3}{2k+4} S_{k-1}(t),
\end{equation*}
which plainly gives
\begin{align*}
\epsilon_{k}^{(j+1)} & =  \frac{k+1}{2k+4} \epsilon_{k+1}^{(j)} + \frac{k+3}{2k+4} \epsilon_{k-1}^{(j)},
\end{align*}
for all $k \geq 0$ and $j \geq 0$. Since we know the values of $\epsilon_k^{(0)}$ for all $k\geq 0$ and $\epsilon_{-1}^{(j)} = 0$ for all $j \geq 0$,  the recursion above completely determines the values of all $\epsilon_k^{(j)}$. A computation leads us to
\begin{equation*}
\frac{1}{4\,\omega_4}{{k+3}\choose{3}}\Lambda_k(\phi_6) = \left\{
\begin{array}{cc}
-\displaystyle\frac{8 ( k+2)}{( k-1) (k+1) ( k+3) (k+5)}\,, &  {\rm for} \ k \ {\rm even};\\
\\
-\displaystyle\frac{8 ( k+1) ( k+3)}{k( k-2) ( k+2) ( k+4) (k+6)}\,,&   {\rm for} \ k \ {\rm odd},
\end{array}
\right.
\smallskip
\end{equation*}
and it follows that $\Lambda_0(\phi_6) , \Lambda_1(\phi_6) >0$ and $\Lambda_k(\phi_6) <0$ if $k\geq 2$, as claimed.

\medskip

\noindent {\it Step 5: Case} $d=7$. The argument here is analogous to the case $d=5$. We start by relabeling 
\begin{equation*}
T_k(t):= C_k^{5/2}(t).
\end{equation*}
By the definition \eqref{Sec6_Def_Gegenbauer} we have
\begin{equation}\label{Sec6_def_T_k}
(1 - 2rt + r^2)^{-5/2} = \sum_{k=0}^{\infty} T_k(t)\, r^k\,,
\end{equation}
and thus $T_k(1) = {{k+4}\choose{4}}$. If we differentiate \eqref{Sec6_Def_Legendre} twice with respect to the variable $t$ and compare with \eqref{Sec6_def_T_k} we find that 
\begin{equation}\label{Sec6_rel_T_P''}
3T_k(t) = P''_{k+2}(t)
\end{equation}
for all $k \geq 0$. From \eqref{Sec6_def_phi_d}, \eqref{Sec6_Def_Lambda_k} and \eqref{Sec6_rel_T_P''} we have
\begin{align*}
\frac{3}{2\,\omega_5}{{k+4}\choose{4}}\Lambda_k(\phi_7) & = \int_{-1}^1 (2-2t)^{3/2} \,P_{k+2}''(t)\,(1+3t+2t^2-2t^3-3t^4-t^5)\,\dt\\
& = \delta_{k+2}^{(0)} + 3\delta_{k+2}^{(1)} + 2\delta_{k+2}^{(2)} - 2\delta_{k+2}^{(3)} - 3\delta_{k+2}^{(4)} - \delta_{k+2}^{(5)},
\end{align*}
where
\begin{equation*}
\delta_{k}^{(j)} :=  \int_{-1}^1 (2-2t)^{3/2} P_{k}''(t)\,t^j\,\dt.
\end{equation*}
The values of $\delta_{k}^{(0)}$, for $k \geq 0$, are given by Proposition \ref{Prop14}. From \eqref{Sec6_Prop14_eq1} we have (recall that we have set $T_{-1}(t) =0$)
\begin{equation*}
t\,T_k(t)=\frac{k+1}{2k+5}T_{k+1}(t)+\frac{k+4}{2k+5} T_{k-1}(t),
\end{equation*}
which gives us
\begin{align*}
\delta_{k+2}^{(j+1)} & =  \frac{k+1}{2k+5} \delta_{k+3}^{(j)} + \frac{k+4}{2k+5} \delta_{k+1}^{(j)},
\end{align*}
for all $k \geq 0$ and $j \geq 0$. Since we also know that $\delta_{0}^{(j)} = \delta_{1}^{(j)}= 0$, for $j \geq 0$, we can explicitly find the values of all the  $\delta_{k}^{(j)}$ from the recursion above. A computation leads to 
 \begin{align*}
& \frac{3}{2\omega_5}{{k+4}\choose{4}}\Lambda_k(\phi_7) \\
& =\frac{245760 ( k+1) ( k+2) ( k+3) ( k+4) (15 - 5 k - k^2) (-3 + 5 k + k^2)}{( 2 k-5) ( 2 k-3) ( 2 k-1) ( 2 k+1) ( 2 k+3) ( 2 k+5) ( 
    2 k+7) ( 2 k+9) ( 2 k+11) ( 2 k+13) ( 2 k+15)},
\end{align*}
and once again one can conclude that $\Lambda_0(\phi_7), \Lambda_1(\phi_7)>0$ and $\Lambda_k(\phi_7)<0$ if $k\geq 2$. 

\medskip

This completes the proof of Lemma \ref{Lem13}.
\end{proof}



\section*{Acknowledgements} 
\noindent The software {\it Mathematica} was used in some of the computations of the proof of Lemma \ref{Lem13}. We would like to thank Marcos Charalambides for clarifying the nature of some results in \cite{Ch}, and to Stefan Steinerberger and Keith Rogers for useful comments on the exposition. We are also thankful to William Beckner, Michael Christ, Damiano Foschi and Christoph Thiele, for helpful discussions during the preparation of this work. E. C. acknowledges support from CNPq-Brazil grants $302809/2011-2$ and $477218/2013-0$, and FAPERJ-Brazil grant $E-26/103.010/2012$. Finally, we would like to thank IMPA - Rio de Janeiro and HCM - Bonn for supporting research visits during the preparation of this work.




\begin{thebibliography}{99}

\bibitem{BBCH}
J. Bennett, N. Bez, A. Carbery and D. Hundertmark,
\newblock Heat-flow monotonicity of Strichartz norms,
\newblock Anal. PDE 2 (2009), no. 2, 147--158.

\bibitem{BR}
N. Bez and K. Rogers,
\newblock A sharp Strichartz estimate for the wave equation with data in the energy space,
\newblock J. Eur. Math. Soc. (JEMS) 15 (2013), no. 3, 805--823.

\bibitem{B}
A. Bulut,
\newblock Maximizers for the Strichartz inequalities for the wave equation,
\newblock Differential Integral Equations 23 (2010), no. 11-12, 1035--1072.

\bibitem{C}
E. Carneiro,
\newblock A sharp inequality for the Strichartz norm,
\newblock  Int. Math. Res. Not. IMRN (2009), no. 16, 3127--3145.

\bibitem{Ch} 
M. Charalambides,
\newblock On restricting Cauchy-Pexider functional equations to submanifolds,
\newblock  Aequationes Math. 86 (2013), no. 3, 231--253.

\bibitem{CS} 
M. Christ and S. Shao,
\newblock Existence of extremals for a Fourier restriction inequality,
\newblock Anal. PDE. 5 (2012), no. 2, 261--312.

\bibitem{CS2} 
M. Christ and S. Shao,
\newblock On the extremizers of an adjoint Fourier restriction inequality,
\newblock Adv. Math. 230 (2012), no. 3, 957--977.

\bibitem{DX} 
F. Dai and Y. Xu,
\newblock {\it Approximation Theory and Harmonic Analysis on Spheres and Balls},
\newblock Springer Monographs in Mathematics, New York, NY, 2013.

\bibitem{FVV}
L. Fanelli, L. Vega and N. Visciglia, 
\newblock On the existence of maximizers for a family of restriction theorems,
\newblock Bull. Lond. Math. Soc. 43 (2011), no. 4, 811--817.

\bibitem{FVV2}
L. Fanelli, L. Vega and N. Visciglia, 
\newblock Existence of maximizers for Sobolev-Strichartz inequalities,
\newblock Adv. Math. 229 (2012), no. 3, 1912--1923.

\bibitem{F2}
D. Foschi,
\newblock Maximizers for the Strichartz inequality, 
\newblock  J. Eur. Math. Soc. (JEMS) 9 (2007), no. 4, 739--774.

\bibitem{F}
D. Foschi,
\newblock Global maximizers for the sphere adjoint Fourier restriction inequality,
\newblock preprint at arxiv.org/abs/1310.2510.

\bibitem{FK}
D. Foschi and S. Klainerman,
\newblock Bilinear space-time estimates for homogeneous wave equations,
\newblock Ann. Sci. \'{E}cole Norm. Sup. (4) 33 (2000), no. 2, 211--274.

\bibitem{HS}
D. Hundertmark and S. Shao, 
\newblock Analyticity of extremizers to the Airy-Strichartz inequality,
\newblock Bull. Lond. Math. Soc. 44 (2012), no. 2, 336--352.

\bibitem{HZ}
D. Hundertmark and V. Zharnitsky, 
\newblock On sharp Strichartz inequalities in low dimensions, 
\newblock Int. Math. Res. Not. IMRN (2006), Art. ID 34080, 1--18.

\bibitem{KM1}
S. Klainerman and M. Machedon, 
\newblock Space-time estimates for null forms and the local existence theorem,
\newblock Comm. Pure Appl. Math. 46  (1993), 1221--1268.

\bibitem{KM2}
S. Klainerman and M. Machedon,
\newblock Remark on Strichartz-type inequalities. With appendices by Jean Bourgain and Daniel Tataru, 
\newblock Int. Math. Res. Not. IMRN (1996), 201--220.

\bibitem{KM3}
S. Klainerman and M. Machedon,
\newblock On the regularity properties of a model problem related to wave maps,
\newblock Duke Math. J. 87 (1997), 553--589.

\bibitem{K}
M. Kunze, 
\newblock On the existence of a maximizer for the Strichartz inequality,
\newblock Comm. Math. Phys. 243 (2003), no. 1, 137--162. 

\bibitem{OS}
D. Oliveira e Silva,
\newblock Extremals for Fourier restriction inequalities: planar curves,
\newblock to appear in J. Anal. Math.

\bibitem{OR}
T. Ozawa and K. Rogers,
\newblock Sharp Morawetz estimates,
\newblock  J. Anal. Math. 121 (2013), 163--175. 

\bibitem{Q}
R. Quilodr\'{a}n, 
\newblock On extremizing sequences for the adjoint restriction inequality on the cone,
\newblock J. Lond. Math. Soc. (2) 87 (2013), no. 1, 223--246.

\bibitem{R} 
D. L. Ragozin,
\newblock Rotation invariant measure algebras on Euclidean space,
\newblock Ind. Univ. Math. Journal 23 (1974), no. 12, 1139--1154.

\bibitem{Ra}
J. Ramos, 
\newblock A refinement of the Strichartz inequality for the wave equation with applications,
\newblock Adv. Math. 230 (2012), no. 2, 649--698. 

\bibitem{Sh}
S. Shao,
\newblock Maximizers for the Strichartz and the Sobolev-Strichartz inequalities for the Schr\"{o}dinger equation,
\newblock Electron. J. Differential Equations (2009), No. 3, 13 pp.

\bibitem{Sh2}
S. Shao,
\newblock The linear profile decomposition for the Airy equation and the existence of maximizers for the Airy-Strichartz inequality,
\newblock Anal. PDE 2 (2009), no. 1, 83--117.

\bibitem{SW} 
E. M. Stein and G. Weiss,
\newblock {\it Introduction to Fourier analysis on Euclidean spaces},
\newblock Princeton Univ. Press, Princeton, NJ, 1971.

\bibitem{T}
T. Tao, 
\newblock Some recent progress on the restriction conjecture, 
\newblock Fourier analysis and convexity, 217--243, {\it Appl. Numer. Harmon. Anal.}, Birkh\"{a}user Boston, Boston, MA, 2004. 

\end{thebibliography}
\end{document}